\documentclass[12pt]{article}
\usepackage[utf8]{inputenc}
\usepackage[toc,page]{appendix}
\usepackage{geometry}
\usepackage{bbm}
\usepackage{amsmath,amsthm,amssymb,amsfonts}
\usepackage{mathrsfs} 
\usepackage[colorlinks, linkcolor=blue]{hyperref}
\newtheorem{theorem}{Theorem}
\newtheorem{lemma}[theorem]{Lemma}
\newtheorem{remark}[theorem]{Remark}

\newtheorem{corollary}[theorem]{Corollary}
\newtheorem{definition}[theorem]{Definition}
\numberwithin{equation}{section}
\numberwithin{theorem}{section}
\usepackage{graphicx}
\usepackage{xcolor}
\usepackage{academicons}
\usepackage{orcidlink}
\usepackage{enumitem,lipsum}
\providecommand{\keywords}[1]
{
  \small	
  \textbf{Keywords } #1
}
\providecommand{\thank}[1]
{
  \small	
  \textbf{Acknowledgment } #1
}

\title{\textbf{Explicit formula for the Benjamin–Ono equation with square integrable and real valued initial data and applications to the zero dispersion limit }}
\date{}
\author{Xi Chen}
\begin{document}

\maketitle

\begin{abstract}
In this paper, we extend the Gérard’s formula for the solution of the Benjamin–Ono equation on the line to square integrable and real valued initial data. Combined with this formula, we also extend the Gérard’s formula for the zero dispersion limit of the Benjamin–Ono equation on the line to more singular initial data. In the derivation of the extension of the formula for the zero dispersion limit, we also find an interesting integral equality, which might be useful in other contexts.
\end{abstract}
\\\\
\thank{The author is currently a PhD student at Institut de Mathématique d’Orsay of Université Paris-Saclay, and he would like to thank his PhD advisor Patrick Gérard for his supervision of this paper. The author would also like to thank Shouda Wang for the discussion on the proof of (\ref{3.07}).}
\tableofcontents{}
\section{Introduction}
\subsection{The Benjamin–Ono equation}
The Benjamin–Ono equation is a nonlinear partial integro-differential equation which describes one-dimensional internal waves in deep water. It was introduced by Benjamin in \cite{2}(see also Davis–Acrivos \cite{3}, Ono \cite{4}). On the line, it reads 
\begin{equation}
\label{0.1}
\begin{aligned}
& \partial_{t} u =\partial_{x}\left(|D| u-u^{2}\right), \quad (t,x) \in \mathbb{R}\times\mathbb{R}, \\  & u(0,x) = u_0.
\end{aligned}
\end{equation}
Here $u = u(t, x)$ denotes a real valued function. We refer to the book by Klein and Saut \cite{10} for a recent survey of this equation. In this paper, we denote by $H_r^s$ (or $L_r^p$ with $p = 2, \infty$) the Sobolev (or Lebesgue) space of real valued functions. \\\\
The global well-posedness of (\ref{0.1}) in $H_r^{s}(\mathbb{R})$ with $s \geq 0$ was proved in \cite{5}\cite{12} by a synthesis of Tao’s gauge transformation \cite{9} and $X^{s, b}$ techniques. In \cite{19}, M. Ifrim and D. Tataru have provided a much simpler proof of the local well-posedness of (\ref{0.1}) in $L_r^2(\mathbb{R})$. Recently, R. Killip, T. Laurens and M. Vişan have proved the global well-posedness of (\ref{0.1}) in $H_r^{s}(\mathbb{R})$ with $-\frac{1}{2} < s < 0$ \cite{7}. The unconditional uniqueness in $H^s(\mathbb{R})$ with $s > 3-\sqrt{33 / 4}$ has been recently proved in \cite{13}.
\begin{theorem}
[\cite{5}, \cite{12}, \cite{19} \cite{13}, \cite{7}]
\label{theorem 1.1}
For every $u_{0} \in H_r^{s}(\mathbb{R})$ with $s > 3-\sqrt{33 / 4}$, there exists a unique solution $u \in C\left(\mathbb{R}, H_{r}^{s}(\mathbb{R})\right)$ of (\ref{0.1}) with $u(0) = u_0$. Also, for every $T > 0$, the flow map $u_{0} \in H^{s}(\mathbb{R}) \mapsto u \in C\left([-T, T], H^{s}(\mathbb{R})\right)$ is continuous. Moreover, this flow map $u_{0} \in H^{s}(\mathbb{R}) \mapsto u \in C\left([-T, T], H^{s}(\mathbb{R})\right)$ can be continuously extended to $H^s(\mathbb{R})$ for any $s > -\frac{1}{2}$.
\end{theorem}
Our aim in this paper is to give an explicit formula of the solution $u(t)$ in terms of any initial data $u_0$ in $L_r^2(\mathbb{R})$. Before presenting our main results, we need to introduce the Lax pair structure for (\ref{0.1}).
\subsection{The Lax pair}
In this paper, we denote by $L_{+}^{2}(\mathbb{R})$ the Hardy space corresponding to $L^2(\mathbb{R})$ functions having a Fourier transform supported in the domain $\xi \geq 0$. Recall that the space $L_{+}^{2}(\mathbb{R})$ identifies to holomorphic functions on the upper-half plane $\mathbb{C}_{+}:=\{z \in \mathbb{C}: \operatorname{Im}(z)>0\}$ such that
\begin{align*}
\sup _{y>0} \int_{\mathbb{R}}|f(x+i y)|^{2} d x<+\infty.
\end{align*}
The Riesz-Szeg\H{o} projector $\Pi$ is the orthogonal projector from $L^2(\mathbb{R})$ onto $L_{+}^2(\mathbb{R})$. It is given by
\begin{equation}
\label{1.10}
\forall f \in L^2(\mathbb{R}), \quad \forall z \in \mathbb{C}_{+}, \quad \Pi f (z) =  \frac{1}{2i\pi} \int_{\mathbb{R}} \frac{f(y)}{y-z} dy.
\end{equation}
The Toeplitz operator on $L_{+}^2(\mathbb{R})$ associated to a function $b \in L^{\infty}(\mathbb{R})$ is defined by
\begin{align*}
T_{b} f : =\Pi(b f), \quad f \in L_{+}^2(\mathbb{R}).
\end{align*}
We notice that for $b \in L_{r}^{\infty}(\mathbb{R})$, $T_b$ is a self-adjoint operator on $L_{+}^2(\mathbb{R})$. As shown in \cite[Proposition 2.2]{18}, we also remark that 
\begin{equation}
\label{1.3}
T_b \in \mathscr{L} \left(L_{+}^2(\mathbb{R})\right) \text{ if and only if } b \in L^{\infty}(\mathbb{R}).
\end{equation}
For $u \in L_{r}^{2}(\mathbb{R})$, the operator $L_u$ is defined by
\begin{align*}
\forall f \in \operatorname{Dom}\left(L_{u}\right)=H_{+}^{1}:=H^{1}(\mathbb{R}) \cap L_{+}^{2}(\mathbb{R}), \quad  L_{u} f :=D f-T_{u} f \text{ with } D:=\frac{1}{i} \frac{d}{d x}.
\end{align*}
We notice that $L_u$ is a semi–bounded selfadjoint operator on $L_{+}^2(\mathbb{R})$.\\\\
Also, we recall the definition of $G$ in \cite{1},
\begin{align*}
\forall f \in \operatorname{Dom}\left(G\right): = \left\{f \in L_{+}^{2}(\mathbb{R}):  \hat{f} \in H^{1}(0, \infty)\right\},\quad  \widehat{G f}(\xi) :=i \frac{d}{d \xi}[\hat{f}(\xi)] \mathbf{1}_{\xi>0}.
\end{align*}
Here $G$ is the adjoint of the operator of multiplication by $x$ on $L_{+}^2(\mathbb{R})$, and we notice that $\left(\operatorname{Dom}\left(G\right), -iG\right)$ is maximally dissipative. We also notice that 
\begin{align*}
\forall f \in \operatorname{Dom}\left(G\right), \left|\hat{f}\left(0^{+}\right)\right|^2 = -4\pi\operatorname{Im}\left\langle G f \mid f\right\rangle \leq 4\pi \|Gf\|_{L^2} \|f\|_{L^2}.
\end{align*}
Therefore, we can define
\begin{align*}
\forall f \in \operatorname{Dom}\left(G\right), \quad I_{+}(f):=\hat{f}\left(0^{+}\right).
\end{align*}
In fact, as observed in \cite[Lemma 3.4]{7}, the resolvent of $G$ is given by 
\begin{equation}
\label{1.20}
\forall z \in \mathbb{C}_{+},\quad \forall f \in L_{+}^2(\mathbb{R}),\quad (G-z \operatorname{Id})^{-1} f(x)=\frac{f(x)-f(z)}{x-z},
\end{equation}
and we have
\begin{equation}
\label{1.30}
\forall z \in \mathbb{C}_{+}, \quad \forall f \in L_{+}^2(\mathbb{R}), \quad f(z)=\frac{1}{2 i \pi} I_{+}\left((G-z \operatorname{Id})^{-1} f\right).
\end{equation}
\subsection{The explicit formula}
The explicit formula for the solution of (\ref{0.1}) with the initial data $u_0 \in L_{r}^{2}(\mathbb{R}) \cap L^{\infty}(\mathbb{R})$ has been introduced by P. Gérard in \cite{1}. In this paper, we extend the explicit formula of the solution to (\ref{0.1}) to the initial data $u_0 \in L_{r}^{2}(\mathbb{R})$.
\begin{theorem}
\label{theorem 1.2}
For $u_0 \in L_{r}^{2}(\mathbb{R})$, let $u \in C\left(\mathbb{R}, L_{r}^{2}(\mathbb{R})\right)$ be the corresponding solution of (\ref{0.1}) in the sense of the continuous extension of the flow map as shown in Theorem \ref{theorem 1.1}. Then $u(t)=\Pi u(t)+\overline{\Pi u}(t)$, with 
\begin{equation}
\label{0.2}
\Pi u(t, z)=\frac{1}{2 i \pi} I_{+}\left(\left(G-2 t L_{u_{0}}-z \mathrm{Id}\right)^{-1} \Pi u_{0}\right), \quad \forall z \in \mathbb{C}_{+},
\end{equation}
where
\begin{align*}
\left(G-2 t L_{u_{0}}-z \mathrm{Id}\right)^{-1}: L_{+}^2(\mathbb{R}) \rightarrow \operatorname{Dom}\left(A_{t}\right)  \text{ is well-defined for every } z \in \mathbb{C}_{+}.
\end{align*}
\end{theorem}
\begin{remark}
In \cite{7}, R. Killip, T. Laurens and M. Vişan have obtained another explicit formula for a Hamiltonian system corresponding to (\ref{0.1})(see \cite[Theorem 6.1]{7} for details). Also, from this formula, they have recovered the formula (\ref{0.2}) with the initial data $u_0 \in L_r^2(\mathbb{R})\cap L^{\infty}(\mathbb{R})$, which has been firstly obtained in \cite{1}.
\end{remark}
\subsection{Zero dispersion limit}
In \cite{6}, P. Gérard considered the Benjamin–Ono equation on the line with a small dispersion $\varepsilon > 0$,
\begin{equation}
\label{1.02}
\begin{aligned}
& \partial_{t} u=\partial_{x}\left(\varepsilon|D| u-u^{2}\right), \quad (t,x) \in \mathbb{R}\times \mathbb{R},\\
& u^{\varepsilon}(0, x)=u_{0}(x).
\end{aligned}
\end{equation}
Observe that the $L^2$ norm of $u^{\varepsilon}(t)$ is independent of $t$, equal to the $L^2$ norm of $u_0$, so there exists a subsequence $\varepsilon_j$ tending to $0$ such that $u_{\varepsilon_j}(t)$ has a weak limit in $L^2(\mathbb{R})$, and we want to show that all these weak limits coincide under certain initial data conditions. If all these weak limits coincide, we call this weak limit the zero dispersion limit. Combined with the explicit formula (\ref{0.2}) for $u_0 \in L_r^2(\mathbb{R}) \cap L^{\infty} (\mathbb{R})$, P. Gérard has obtained the explicit formula for the zero dispersion limit in \cite{6}. In this paper, since we have obtained the explicit formula (\ref{0.2}) with $u_0 \in L_r^2(\mathbb{R})$, we can extend the explicit formula for the zero dispersion limit to more singular initial data.
\begin{theorem}
\label{theorem 1.3}
Let $u_0 \in L_r^2(\mathbb{R}) \cap L_{loc}^{\infty} (\mathbb{R})$ with $\lim_{x \to \infty} \frac{|u_0(x)|}{|x|} = 0$. Then for every $ t\in \mathbb{R}$, the corresponding solution $u^{\varepsilon}(t)$ to (\ref{1.02}) converges weakly in $L^{2}(\mathbb{R})$ to $Z D\left[u_{0}\right](t)$, characterized by
\begin{align*}
\forall x \in \mathbb{R}, \quad Z D\left[u_{0}\right](t, x)=\Pi Z D\left[u_{0}\right](t, x)+\overline{\Pi Z D\left[u_{0}\right](t, x)}
\end{align*}
and
\begin{equation}
\label{1.4}
\begin{aligned}
\forall z \in \mathbb{C}_{+}, \quad \Pi Z D\left[u_{0}\right](t, z) & =\frac{1}{2 i \pi} I_{+}\left(\left(G+2 t T_{u_{0}}-z \mathrm{Id}\right)^{-1} \Pi u_{0}\right) \\ &= \frac{1}{4i\pi t}\int_{\mathbb{R}} {\rm Log }\left(1 + \frac{2t u_0(y)}{y-z}\right) dy,
\end{aligned}
\end{equation}
where ${\rm Log }$ denotes the principal value of the logarithm, and
\begin{align*}
\left(G+2 t T_{u_{0}}-z \mathrm{Id}\right)^{-1}: L_{+}^2(\mathbb{R}) \rightarrow \operatorname{Dom}(G)  \text{ is well-defined for every } z \in \mathbb{C}_{+}.
\end{align*}
\end{theorem}
\begin{remark}
We observe that $u_0 \in L_r^2(\mathbb{R})$ with $|u_0(x)| \leq C \langle x \rangle^{k}(k<1)$ satisfies the initial data condition in Theorem \ref{theorem 1.3}, so we can give the formula of the zero dispersion limit for every $t \in \mathbb{R}$ with such an initial datum.
\end{remark}
\begin{remark}
\label{remark 1.6}
In \cite{6}, P. Gérard has also obtained the following description of the zero dispersion limit: Assume that the initial data $u_0 \in L_r^2(\mathbb{R}) \cap C^1(\mathbb{R})$ with $|u_0(x)|+|u_0^{\prime}(x)| \rightarrow 0$, then for every $t \in \mathbb{R}$, the set $K_t(u_0)$ of critical values of the function
\begin{align*}
y \in \mathbb{R} \mapsto y+2 t u_0(y)
\end{align*}
is a compact subset of measure 0. For every connected component $\Omega$ of $K_t\left(u_0\right)^c$, there exists a nonnegative integer $\ell$ such that, for every $x \in \Omega$, the equation
\begin{align*}
y+2 t u_0(y)=x
\end{align*}
has $2\ell+1$ simple real solutions
\begin{align*}
y_0(t, x)<y_1(t, x)<\cdots<y_{2 \ell}(t, x),
\end{align*}
and the zero dispersion limit is given by
\begin{equation}
\label{1.9}
Z D\left[u_0\right](t, x)=\sum_{k=0}^{2 \ell}(-1)^k u_0\left(y_k(t, x)\right).
\end{equation}
Formula (\ref{1.9}) was proved by Miller-Wetzel \cite{16}(see also Miller-Xu \cite{17}) in the special case of a rational Klaus–Shaw initial potential, and by L. Gassot \cite{14}\cite{15} in the special case of a general bell shaped initial potential with periodic boundary conditions.
\end{remark}
\begin{remark}
In \cite{6}, P. Gérard has obtained (\ref{1.4}) with the initial data $u_0 \in L_r^2(\mathbb{R}) \cap L^{\infty} (\mathbb{R})$. In the derivation of the second equality in (\ref{1.4}), P. Gérard first considered the rational initial data to deduce this equality, and then extend this equality to $u_0 \in L_r^2(\mathbb{R}) \cap L^{\infty} (\mathbb{R})$. However, this proof is not a direct derivation. In this paper, we provide a direct proof of the second equality of (\ref{1.4}), and this direct approach allows us to extend this equality to $u_0 \in L_r^2(\mathbb{R}) \cap L_{loc}^{\infty} (\mathbb{R})$ with $\lim_{x \to \infty} \frac{|u_0(x)|}{|x|} = 0$. 
\end{remark}
In the direct derivation of the second equality of (\ref{1.4}), we also find an interesting integral equality (\ref{3.006}), which might be useful in other contexts. We summarize this interesting equality in the following lemma.
\begin{lemma}
\label{lemma 1.7}
For $f \in L^2(\mathbb{R}) \cap L^{\infty}(\mathbb{R})$ and $n \in \mathbb{N}_{\geq 1}$, we have
\begin{equation}
\label{3.006}
\begin{aligned}
& \int_{\mathbb{R}^n} f(y_1) f(y_2-y_1) ... f(y_n - y_{n-1})f(-y_n) dy_1 dy_2...dy_n \\  = & (n+1) \int_{\{\forall 1\leq j \leq n, y_j > 0\}} f(y_1) f(y_2-y_1) ... f(y_n - y_{n-1}) f(-y_n)dy_1 dy_2...dy_n.
\end{aligned}
\end{equation}
\end{lemma}
With a slight modification of the proof of Theorem \ref{theorem 1.3}, we can obtain the following zero dispersion limit result for $u_0 \in L_r^2(\mathbb{R})$ with $|u_0(x)| \leq C \langle x\rangle$ in a short time.
\begin{corollary}
\label{corollary 1.4}
Let $u_0 \in L_r^2(\mathbb{R})$ with $|u_0(x)| \leq C \langle x\rangle$. Then for every $ |t| < \frac{1}{2C}$, the corresponding solution $u^{\varepsilon}(t)$ to (\ref{1.02}) converges weakly in $L^{2}(\mathbb{R})$ to $Z D\left[u_{0}\right](t)$, characterized by
\begin{align*}
\forall x \in \mathbb{R}, \quad Z D\left[u_{0}\right](t, x)=\Pi Z D\left[u_{0}\right](t, x)+\overline{\Pi Z D\left[u_{0}\right](t, x)}
\end{align*}
and
\begin{equation}
\label{1.41}
\begin{aligned}
\forall z \in \mathbb{C}_{+}, \quad \Pi Z D\left[u_{0}\right](t, z) & =\frac{1}{2 i \pi} I_{+}\left(\left(G+2 t T_{u_{0}}-z \mathrm{Id}\right)^{-1} \Pi u_{0}\right) \\ &= \frac{1}{4i\pi t}\int_{\mathbb{R}} {\rm Log }\left(1 + \frac{2t u_0(y)}{y-z}\right) dy,
\end{aligned}
\end{equation}
where ${\rm Log }$ denotes the principal value of the logarithm, and
\begin{align*}
\left(G+2 t T_{u_{0}}-z \mathrm{Id}\right)^{-1}: L_{+}^2(\mathbb{R}) \rightarrow \operatorname{Dom}(G)  \text{ is well-defined for every } z \in \mathbb{C}_{+}.
\end{align*}
\end{corollary}
\begin{remark}
\label{remark 1.70}
Even for $u_0 \in L_r^2(\mathbb{R})$, we know that $\frac{2tu_0}{y-z} \notin \mathbb{R}$ for all $z \in \mathbb{C}_{+}$ and for all $t \in \mathbb{R}$, so ${\rm Log }\left(1 + \frac{2t u_0(y)}{y-z}\right)$ is well defined in $\mathbb{C}_{+}$. We also notice that
\begin{align*}
\frac{1}{4i\pi t}\int_{\mathbb{R}}{\rm Log }\left(1 + \frac{2t u_0(y)}{y-z}\right) dy=\frac{1}{2i\pi}\int_{\mathbb{R}}\int_{0}^1 \frac{u_0(y)}{y-z+2stu_0(y)}dsdy,
\end{align*}
since 
\begin{align*}
\frac{1}{y-z+2stu_0(y)} \in L_{s}^{\infty}(0,1)L_y^2(\mathbb{R})\cap  L_{s}^{\infty}(0,1)L_y^{\infty}(\mathbb{R}),
\end{align*}
we can deduce that $\frac{1}{4i\pi t}\int_{\mathbb{R}}{\rm Log }\left(1 + \frac{2t u_0(y)}{y-z}\right) dy$ is well defined and holomorphic in $\mathbb{C}_{+}$. This tells us the formula for the zero dispersion limit
\begin{align*}
\Pi Z D\left[u_0\right](t, z) = \frac{1}{4 i \pi t} \int_{\mathbb{R}}{\rm Log } \left(1+\frac{2 t u_0(y)}{y-z}\right) d y
\end{align*}
might be extended to $u_0 \in L_r^2(\mathbb{R})$ for every $t \in \mathbb{R}$, but the difficulty lies in solving the problem of switching the order of a double limit, see also Section \ref{section 4} for details.
\end{remark}
\subsection{Maximally dissipative operators and the Kato-Rellich theorem}
In this paper, we mainly apply the Kato-Rellich theorem to show that operators remain maximally dissipative after some perturbations. To present the Kato-Rellich theorem for maximally dissipative operators, we first need to introduce the following definition of the dissipative and maximally dissipative operators in Hilbert spaces.
\begin{definition}
Let $(D(A), A)$ be an operator in a Hilbert space $\mathscr{H}$.\\\\
1. We say that $A$ is dissipative if for all $g \in D(A)$ and all $\lambda >0$, 
\begin{align*}
\|(\lambda I- A) g\| \geq \lambda \|g\| .
\end{align*}
2. We say that $A$ is maximal dissipative if it is dissipative and for all $h \in \mathscr{H}$ and for all $\lambda > 0$, there exists $g \in D(A)$ such that $(\lambda I - A)g = h$.
\end{definition}
\begin{remark}
\label{remark 1.9}
In fact, an operator $(D(A), A)$ in a Hilbert space $\mathscr{H}$ is dissipative if and only if for all $g \in D(A), \, \Re \langle Ag, g\rangle \leq 0$.
\end{remark}
\begin{remark}
Let $(D(A), A)$ be a maximally dissipative operator in a Hilbert space $\mathscr{H}$. From the definition of maximally dissipative operators, we can deduce that, for all $\lambda >0$, we have
\begin{equation}
\label{1.90}
\|(\lambda I -A)^{-1}\|_{\mathscr{L}(\mathscr{H})} \leq \frac{1}{\lambda}, \qquad \|A(\lambda I -A)^{-1}\|_{\mathscr{L}(\mathscr{H})} \leq 1.
\end{equation}
\end{remark}
Since the Kato-Rellich theorem involves the related notion of the perturbation of operators, we give the following definition of the relative bound of an operator with respect to another operator (see also the definition in \cite{11}). 
\begin{definition}
Let $\left(D(A), A\right)$ and $\left(D(B), B\right)$ be densely defined linear operators on a Hilbert space $\mathscr{H}$. Suppose that:\\\\
(i) $D(A) \subset D(B)$; \\
(ii) For some $a$ and $b$ in $\mathbb{R}$ and all $\varphi \in D(A)$,
\begin{align*}
\|B \varphi\| \leq a\|A \varphi\|+b\|\varphi\|.
\end{align*}
Then $B$ is said to be $A$-bounded. The infimum of such $a$ is called the relative bound of $B$ with respect to $A$. If the relative bound is $0$, we say that $B$ is infinitesimally small with respect to $A$.
\end{definition}
Then we state the Kato-Rellich theorem for maximally dissipative operators.
\begin{theorem}
[Kato-Rellich theorem]
\label{theorem 1.5}
Let $\left(D(A), A\right)$ be a maximally dissipative operator which is densely defined on a Hilbert space $\mathscr{H}$ and assume $\left(D(B), B\right)$ to be dissipative and $A$-bounded with the relative bound smaller than 1. Then $\left(D(A), A+B\right)$ is also a maximally dissipative operator.
\end{theorem}
We refer to \cite[Theorem X.12]{11} for the proof of the Kato-Rellich theorem for self-adjoint operators. The readers can also see the proof of Corollary \ref{corollary 2.2}.
\subsection{Structure of the paper}
In Section \ref{section 2}, for $u_0 \in L_r^2(\mathbb{R})$, we apply the Kato-Rellich theorem \ref{theorem 1.5} to show that $\left(G-2 t L_{u_0}-z \mathrm{Id}\right)^{-1}$  is well-defined on $L_{+}^2(\mathbb{R})$ for every $z \in \mathbb{C}_{+}$, then we can extend the explicit formula (\ref{0.2}) to $u_0 \in L_r^2(\mathbb{R})$ and prove Theorem \ref{theorem 1.2}.\\\\
In Section \ref{section 3}, for  $u_0 \in L_r^2(\mathbb{R}) \cap L_{loc}^{\infty} (\mathbb{R})$ with $\lim_{x \to \infty} \frac{|u_0(x)|}{|x|} = 0$, we can still apply the Kato-Rellich theorem \ref{theorem 1.5} to show that $\left(G+2 t T_{u_0}-z \mathrm{Id}\right)^{-1}$  is well-defined on $L_{+}^2(\mathbb{R})$ for every $z \in \mathbb{C}_{+}$. Also, we prove Lemma \ref{lemma 1.7} and then adapt the equality (\ref{3.006}) to prove the second equality of (\ref{1.4}). Finally, we show that the zero dispersion limit exists and complete the proof of Theorem \ref{theorem 1.3}.\\\\
In Section \ref{section 4}, we discuss the difficulties in further extensions of the explicit formula (\ref{0.2}) and of the formula (\ref{1.4}) for the zero dispersion limit. We also introduce briefly the results and the open problem on the zero dispersion limit for the Benjamin–Ono equation on the torus. 
\section{Proof of the extension of the explicit formula}
\label{section 2}
In this section, we will show why the formula (\ref{0.2}) can be extended to the initial data in $L_r^2(\mathbb{R})$. In fact, P. Gérard proved directly the formula (\ref{0.2}) for $u_0 \in H_r^2(\mathbb{R})$ in \cite{1}, and then he extended this formula to $u_0 \in L_{r}^{2}(\mathbb{R}) \cap L^{\infty}(\mathbb{R})$.  Let us firstly recall the sketch of proof of the generalized formula for $u_0 \in L_{r}^{2}(\mathbb{R}) \cap L^{\infty}(\mathbb{R})$. We consider the following operator
\begin{align*}
A_{t}:=-iG + 2 i t D, \text{ with }  \operatorname{Dom}\left(A_{t}\right) :=\left\{f \in L_{+}^{2}(\mathbb{R}): \mathrm{e}^{i t \xi^{2}} \hat{f} \in H^{1}(0, \infty)\right\}.
\end{align*}
In fact, we observe that
\begin{align*}
A_t = -iG + 2 i t D = \mathrm{e}^{-it D^2} (-iG )\mathrm{e}^{it D^2},
\end{align*}
so we can easily deduce that $\left( \operatorname{Dom}\left(A_{t}\right), A_{t}\right)$ is maximally disspative. Then, for $u_0 \in L_{r}^{2}(\mathbb{R}) \cap L^{\infty}(\mathbb{R})$, we know that $\left(L_{+}^2(\mathbb{R}), T_{u_0}\right)$ is a bounded and self-adjoint operator, so by a classical perturbation theory, we can deduce that $A_{t}-2 i t T_{u_{0}} = -iG+2 it L_{u_{0}}$ is also maximally dissipative, and then by approximation, we conclude that (\ref{0.2}) holds for $u_0 \in L_{r}^{2}(\mathbb{R}) \cap L^{\infty}(\mathbb{R})$. \\\\
However, for $u_0 \in L_{r}^{2}(\mathbb{R})$, we cannot expect that $T_{u_0}$ to remain bounded and dissipative on $L_{+}^2(\mathbb{R})$, so we cannot adapt directly the argument in \cite{1} in this case. \\\\
Fortunately, we can adapt another approach to verify the formula (\ref{0.2}) for $u_0 \in L_{r}^{2}(\mathbb{R})$. In fact, we observe that for $f \in \operatorname{Dom}\left(A_{t}\right)$, we have
\begin{align*}
-iG f + 2it L_{u_0} f = A_{t} f-2 i t T_{u_{0}} f = \mathrm{e}^{-it D^2} \left( -iG - 2it \mathrm{e}^{it D^2} T_{u_{0}} \mathrm{e}^{-it D^2}\right) \mathrm{e}^{it D^2} f.
\end{align*}
Then we consider the operator 
\begin{align*}
\mathcal{G}_{t} : = -i G - 2it \mathrm{e}^{it D^2} T_{u_{0}} \mathrm{e}^{-it D^2}
\end{align*}
with
\begin{align*}
\operatorname{Dom}\left(\mathcal{G}_{t}\right) = \operatorname{Dom}\left(G\right) : = \left\{f \in L_{+}^{2}(\mathbb{R}):  \hat{f} \in H^{1}(0, \infty)\right\}.
\end{align*}
We recall that $\left(\operatorname{Dom}\left(G\right), -i G\right)$ is maximally dissipative. Now we are going to prove that $B_{u_0}^t : = -2it\mathrm{e}^{it D^2} T_{u_{0}} \mathrm{e}^{-it D^2}$ with the domain $\operatorname{Dom}\left(G\right)$ is dissipative and is infinitesimally small with respect to $G$, and then we can apply Theorem \ref{theorem 1.5} to show that $\left(\operatorname{Dom}\left(G\right), \mathcal{G}_{t}\right)$ is maximally dissipative, and so is $\left(\operatorname{Dom}\left(A_t\right), -iG + 2it L_{u_0}\right)$.
\begin{lemma}
\label{lemma 3.1}
Given $u_0 \in L_r^2(\mathbb{R})$, for any $t \in \mathbb{R}$, the operator $B_{u_0}^t := -2it\mathrm{e}^{it D^2} T_{u_{0}} \mathrm{e}^{-it D^2}$ with the domain $\operatorname{Dom}\left(G\right)$ is dissipative and is infinitesimally small with respect to $G$.
\end{lemma}
\begin{proof}
 Firstly, we want to show that $B_{u_0}^t$ is well-defined on $\operatorname{Dom}(G)$ and $\left(\operatorname{Dom}\left(G\right), B_{u_0}^t\right) $ is dissipative. In fact, if we can show that 
\begin{equation}
\label{3.1}
\forall f \in \operatorname{Dom}\left(G\right) \text{ and }\forall 0<t<\infty, \quad t^{\frac{1}{2}} \mathrm{e}^{-it D^2} f \in L^{\infty}(\mathbb{R}),
\end{equation}
since $u_0 \in L_r^2(\mathbb{R})$, we can infer that $B_{u_0}^t$ is well-defined on $\operatorname{Dom}(G)$. Also, from (\ref{3.1}), 
we can infer that for all $ f \in \operatorname{Dom}\left(G\right)$ and for all $0<t<\infty$, we have
\begin{align*}
\Re \left\langle -2it \mathrm{e}^{it D^2} T_{u_0} \mathrm{e}^{-it D^2}f,  f\right\rangle & =  2 \Im \left\langle T_{u_0} t^{\frac{1}{2}}\mathrm{e}^{-it D^2} f, t^{\frac{1}{2}}\mathrm{e}^{-it D^2} f\right\rangle \\ & = 2\Im \left\langle  t^{\frac{1}{2}}\mathrm{e}^{-it D^2} f, T_{u_0} t^{\frac{1}{2}}\mathrm{e}^{-it D^2} f\right\rangle \\ & = - \Re \left\langle  \mathrm{e}^{-it D^2} f, - 2it T_{u_0} \mathrm{e}^{-it D^2} f\right\rangle,
\end{align*}
which implies that
\begin{align*}
\forall f \in \operatorname{Dom}\left(G\right) \text{ and }\forall 0<t<\infty, \quad \Re \left\langle B_{u_0}^t f, f\right\rangle = 0.
\end{align*}
Then from Remark \ref{remark 1.9}, we can deduce that $\left(\operatorname{Dom}\left(G\right), B_{u_0}^t\right) $ is dissipative. So the point is to prove (\ref{3.1}).\\\\
Before proving (\ref{3.1}), we define a function $g \in L_{+}^2(\mathbb{R})$ by
\begin{align*}
\widehat{g}(\xi) : = \mathbf{1}_{\xi \geq 0} \, \mathrm{e}^{-\xi}
\end{align*}
with
\begin{align*}
I_{+} (g) = 1.
\end{align*}
We recall that
\begin{align*}
|I_{+}(f)|^2 = -4\pi\operatorname{Im}\left\langle G f \mid f\right\rangle \leq 4\pi \|Gf\|_{L^2} \|f\|_{L^2}.
\end{align*}
Then we have 
\begin{align*}
\left\|t^{\frac{1}{2}} \mathrm{e}^{-it D^2} f \right\|_{L^{\infty}} & \leq |t|^{\frac{1}{2}}\left\| \mathrm{e}^{-it D^2}\left(f-I_{+}(f)g\right) \right\|_{L^{\infty}} + |t|^{\frac{1}{2}}\left\|\mathrm{e}^{-it D^2} \left(I_{+} (f) g\right) \right\|_{L^{\infty}} \\ & : = I_1 + I_2.
\end{align*}
By the Hausdorff–Young inequality, we have the following estimate for $I_2$,
\begin{align*}
I_2 \leq |t|^{\frac{1}{2}} |I_{+}(f)| \|\widehat{g}\|_{L_{\xi}^1} \leq  C |t|^{\frac{1}{2}} \|Gf\|_{L^2}^{\frac{1}{2}} \|f\|_{L^2}^{\frac{1}{2}}.
\end{align*}
For $I_1$, from the dispersive estimate, we have
\begin{align*}
I_1 & \leq |t|^{\frac{1}{2}} \left\|\mathrm{e}^{-it D^2}\left(f-I_{+}(f)g\right)\right\|_{L^{\infty}} \\ & \leq C \left\|f-I_{+}(f)g\right\|_{L^1}\\ & \leq C \left\|f-I_{+}(f)g\right\|_{L^2}^{\frac{1}{2}}\left\|x\left(f-I_{+}(f)g\right)\right\|_{L^2}^{\frac{1}{2}} \\ & \leq C \left\|f-I_{+}(f)g\right\|_{L^2}^{\frac{1}{2}}\left\|G\left(f-I_{+}(f)g\right)\right\|_{L^2}^{\frac{1}{2}}.
\end{align*}
Here, $x\left(f-I_{+}(f)g\right) = G\left(f-I_{+}(f)g\right)$ since we have $\mathbf{1}_{\xi\geq 0}\left(\widehat{f}(\xi) - I_{+}(f)\widehat{g}(\xi)\right)$ is continuous at $\xi = 0$.\\\\
Then we have
\begin{align*}
\left\|f-I_{+}(f)g\right\|_{L^2} &\leq \left\|f\right\|_{L^2} + |I_{+}(f)|\|g\|_{L^2} \\ & \leq \left\|f\right\|_{L^2} + C \|Gf\|_{L^2}^{\frac{1}{2}} \|f\|_{L^2}^{\frac{1}{2}}
\end{align*}
and
\begin{align*}
\left\|G\left(f-I_{+}(f)g\right)\right\|_{L^2} &\leq \left\|G f\right\|_{L^2} + |I_{+}(f)|\|G g\|_{L^2} \\ & \leq \left\|Gf\right\|_{L^2} + C \|Gf\|_{L^2}^{\frac{1}{2}} \|f\|_{L^2}^{\frac{1}{2}}.
\end{align*}
From the above estimates for $I_1$ and $I_2$, we can deduce that
\begin{equation}
\label{3.2}
\begin{aligned}
\left\|t^{\frac{1}{2}} \mathrm{e}^{-it D^2} f \right\|_{L^{\infty}} & \leq C |t|^{\frac{1}{2}} \|Gf\|_{L^2}^{\frac{1}{2}} \|f\|_{L^2}^{\frac{1}{2}}+ C \|Gf\|_{L^2}^{\frac{1}{4}} \|f\|_{L^2}^{\frac{3}{4}}\\ & + C \|Gf\|_{L^2}^{\frac{3}{4}} \|f\|_{L^2}^{\frac{1}{4}},
\end{aligned}
\end{equation}
which verifies (\ref{3.1}).\\\\
Then we prove that $B_{u_0}^t$ is infinitesimally small with respect to $G$. It is equivalent to show the following argument: Let $0<t< \infty$ (fixed) and $u_0 \in L_r^2(\mathbb{R})$, for any $\varepsilon > 0$, we have
\begin{equation}
\label{3.3}
 \|B_{u_0}^t f \|_{L^2} \leq \varepsilon \|Gf\|_{L^2} + C_{\varepsilon} \|f\|_{L^2}, \quad \forall f \in \operatorname{Dom}\left(G\right).
\end{equation}
In fact, we can combine (\ref{3.2}) with the Young's inequality for products and then we can deduce (\ref{3.3}). The proof of Lemma \ref{lemma 3.1} is complete.
\end{proof}
With Lemma \ref{lemma 3.1}, we can now adapt directly Theorem \ref{theorem 1.5} to show that $\left(\operatorname{Dom}\left(G\right), \mathcal{G}_{t}\right)$ is maximally dissipative, and so is $\left(\operatorname{Dom}\left(A_t\right), -iG + 2it L_{u_0}\right)$. For the readers' convenience, we reproduce the proof of Theorem \ref{theorem 1.5} in the proof of Corollary \ref{corollary 2.2}.
\begin{corollary}
\label{corollary 2.2}
Let $u_0 \in L_r^2(\mathbb{R})$, for any $t \in \mathbb{R}$, $\left(\operatorname{Dom}\left(G\right), \mathcal{G}_{t}\right)$ and $\left(\operatorname{Dom}\left(A_t\right), -iG + 2it L_{u_0}\right)$ are maximally dissipative. 
\end{corollary}
\begin{proof}
We recall that
\begin{align*}
-iG + 2it L_{u_0}=A_{t} f-2 i t T_{u_{0}} = \mathrm{e}^{-it D^2} \mathcal{G}_t \mathrm{e}^{it D^2},
\end{align*}
so we only need to show that  $\left(\operatorname{Dom}\left(G\right), \mathcal{G}_{t}\right)$ is maximally dissipative. \\\\
Since $\left(\operatorname{Dom}\left(G\right), -iG\right)$ and $\left(\operatorname{Dom}\left(G\right), B_{u_0}^t\right)$ are dissipative, we know that $\left(\operatorname{Dom}\left(G\right), \mathcal{G}_{t}\right)$ is dissipative. \\\\
Then we only need to show that $\mathcal{G}_{t}+i z \operatorname{Id}: \operatorname{Dom}\left(G\right) \rightarrow L_{+}^{2}(\mathbb{R})$ is bijective for some $z \in \mathbb{C}_{+}$.  We write
\begin{align*}
\mathcal{G}_t + i z \operatorname{Id}  =  -iG + B_{u_0}^t + i z \operatorname{Id} = \left(\operatorname{Id}  + B_{u_0}^t (-iG+i z \operatorname{Id})^{-1}\right)\left(-iG+i z \operatorname{Id}\right).
\end{align*}
Since $B_{u_0}^t$ is infinitesimally small with respect to $G$, for $z \in i\mathbb{R}_{>0}$ and for any $\varepsilon >0$, we have
\begin{align*}
\forall f \in L_{+}^2(\mathbb{R}), \quad \left\|B_{u_0}^t(-iG+ iz\operatorname{Id})^{-1}f\right\|_{L^2} & \leq \varepsilon \left\|G(-iG+ iz\operatorname{Id})^{-1} f\right\|_{L^2} + C_{\varepsilon} \left\|(-iG+ iz\operatorname{Id})^{-1} f\right\|_{L^2} \\ & \leq \left(\varepsilon+\frac{C_{\varepsilon}}{\Im(z)}\right) \|f\|_{L^2}.
\end{align*}
The last inequality above comes from (\ref{1.90}).\\\\
Then we choose $\varepsilon = \frac{1}{4}$ and $z \in i\mathbb{R}_{>0}$ such that $C_{\frac{1}{4}}/\Im(z) < \frac{1}{4}$, and we have
\begin{align*}
\left\|B_{u_0}^t(G+ iz\operatorname{Id})^{-1}f\right\|_{L^2} < \frac{1}{2}\|f\|_{L^2}.
\end{align*}
Since $G$ is maximally dissipative, we can deduce that $\mathcal{G}_{t}+i z \operatorname{Id}: \operatorname{Dom}\left(G\right) \rightarrow L_{+}^{2}(\mathbb{R})$ is bijective for some $z \in i\mathbb{R}_{>0}$, which provides that $\left(\operatorname{Dom}\left(G\right), \mathcal{G}_{t}\right)$ is maximally dissipative, so is $\left(\operatorname{Dom}\left(A_t\right), -iG + 2it L_{u_0}\right)$.
\end{proof}
By Corollary \ref{corollary 2.2},  we know that $\left(\operatorname{Dom}\left(A_t\right), -iG + 2it L_{u_0}\right)$ is maximally dissipative, thus for every $z \in \mathbb{C}_{+}$, the operator $\left(G-2tL_{u_0} - zId\right)^{-1}$ is well-defined on $L_{+}^2(\mathbb{R})$.\\\\
Now we are able to present the following proof of Theorem \ref{theorem 1.2}.\\\\
\textit {Proof of Theorem \ref{theorem 1.2}}. For $u_0 \in L_r^2(\mathbb{R})$, we can take $u_0^n \in L_r^2(\mathbb{R})\cap L^{\infty}(\mathbb{R})$ which tends to $u_0$ in $L^2(\mathbb{R})$, then we can easily deduce that $\Pi u_0^n$ tends to $\Pi u_0$ in $L_{+}^2(\mathbb{R})$. We denote the solutions of (\ref{0.1}) by $u^n(t)$ and $u(t)$ corresponding to $u_0^n$ and $u_0$. By the continuity of the flow map, we can deduce that $u^n(t)$ tends to $u(t)$ in $L^2(\mathbb{R})$. Then for $z \in \mathbb{C}_{+}$, we have
\begin{align*}
\left|\Pi u^n(t,z)-\Pi u(t,z)\right| \leq \int_{0}^{\infty} \mathrm{e}^{-\xi\Im(z)} |\widehat{u^n}(t,\xi)-\widehat{u}(t,\xi)| d\xi \leq C \left\|u^n(t) - u(t)\right\|_{L^2(\mathbb{R})} \rightarrow 0,
\end{align*}
which implies the pointwise convergence of $\Pi u^n(t,z)$ to $\Pi u(t,z)$ for all $z \in \mathbb{C}_{+}$. Moreover, by Lemma \ref{lemma 3.1} and Corollary \ref{corollary 2.2}, we can easily deduce that for every $f \in L_{+}^2(\mathbb{R})$,
\begin{align*}
B_{u_0^n}^t \left(G-zId\right)^{-1} f \rightarrow B_{u_0}^t \left(G-zId\right)^{-1} f \text{ in } L_{+}^2(\mathbb{R}), \quad \forall z \in \mathbb{C}_{+},
\end{align*}
which implies that
\begin{align*}
\left(G-2 t L_{u_{0}^n}-z \mathrm{Id}\right)^{-1}f \rightarrow \left(G-2 t L_{u_{0}}-z \mathrm{Id}\right)^{-1}f  \text{ in } L_{+}^2(\mathbb{R}),  \quad \forall z \in \mathbb{C}_{+}.
\end{align*}
Then we recall the following explicit formula of $\Pi(u^n(t,z))$,
\begin{equation}
\label{2.4}
\Pi(u^n(t,z)) = \frac{1}{2 i \pi} I_{+}\left(\left(G-2 t L_{u_{0}^n}-z\mathrm{Id}\right)^{-1}\Pi u_{0}^n\right), \quad \forall z \in \mathbb{C}_{+}.
\end{equation}
From the previous arguments, we can conclude that the formula (\ref{2.4}) converges pointwisely in $\mathbb{C}_{+}$ to
\begin{align*}
\Pi(u(t,z)) = \frac{1}{2 i \pi} I_{+}\left(\left(G-2 t L_{u_{0}}-z\mathrm{Id}\right)^{-1}\Pi u_{0}\right), \quad \forall z \in \mathbb{C}_{+}.
\end{align*}
The proof is complete.
\section{Proof of the extension of the formula for the zero dispersion limit}
\label{section 3}
In this section, we will show why the formula (\ref{1.4}) can be extended 
to the initial data $u_0 \in L_r^2(\mathbb{R}) \cap L_{loc}^{\infty} (\mathbb{R})$ with $\lim_{x \to \infty} \frac{|u_0(x)|}{|x|} = 0$. Before proving Theorem \ref{theorem 1.3}, let us first give two important observations.  \\\\
First we give an important lemma which is useful in the sequel.
\begin{lemma}
\label{lemma 3.001}
Given $u_0 \in L_r^2(\mathbb{R})$ and $f_0 \in L_{+}^2(\mathbb{R})$, and we define
\begin{align*}
f(t,z) : = \frac{1}{2 i \pi} I_{+}\left(\left(G-2 t L_{u_{0}}-z\mathrm{Id}\right)^{-1}f_0\right), \quad \forall z \in \mathbb{C}_{+}, \, \forall t \in \mathbb{R}.
\end{align*}
Then we have
\begin{equation}
\label{3.001}
\|f\|_{L^2(\mathbb{R})} \leq  \|f_0\|_{L^2(\mathbb{R})}, \quad \forall t \in \mathbb{R}.
\end{equation}
\end{lemma}
\begin{proof}
We first assume $u_0 \in H_r^2(\mathbb{R})$. We follow the approach in the derivation of explicit formula in \cite{1}. In fact,
\begin{align*}
f(t,z) & = \frac{1}{2 i \pi} I_{+}\left(\left(G-2 t L_{u_{0}}-z\mathrm{Id}\right)^{-1}f_0\right) \\ & = \lim _{\varepsilon \rightarrow 0} \frac{1}{2 i \pi}\left\langle\left(G-2 t L_{u_0}-z \mathrm{Id}\right)^{-1} f_0 \mid \chi_{\varepsilon}\right\rangle \\ & = \lim _{\varepsilon \rightarrow 0} \frac{1}{2 i \pi}\left\langle\left(\mathrm{e}^{i t L_{u_0}^2} G \mathrm{e}^{-i t L_{u_0}^2}-2 t L_{u_0}-z \mathrm{Id}\right)^{-1} \mathrm{e}^{i t L_{u_0}^2} \Pi u_0 \mid \mathrm{e}^{i t L_{u_0}^2} \chi_{\varepsilon}\right\rangle,
\end{align*}
where
\begin{align*}
\chi_{\varepsilon}(x):=\frac{1}{1-i \varepsilon x},
\end{align*}
and $\left\langle \cdot \mid \cdot \right\rangle$ is the $L_{+}^2$ inner product. \\\\
We then introduce the family $U(t)$ of unitary operators defined by the linear initial value problem in $\mathscr{L}\left(L_{+}^2(\mathbb{R})\right)$,
\begin{align*}
U^{\prime}(t)=B_{u(t)} U(t), U(0)=\mathrm{Id} \text{ with } B_u:=i\left(T_{|D| u}-T_u^2\right),
\end{align*}
where $u(t)$ is the corresponding solution to \eqref{0.1} with the initial data $u_0$.\\\\
From the proof in \cite[Section 3]{1}, we know that
\begin{align*}
U(t)^* G U(t) =\mathrm{e}^{i t L_{u_0}^2} G \mathrm{e}^{-i t L_{u_0}^2}-2 t L_{u_0} 
\end{align*}
and
\begin{align*}
U(t)^* \chi_{\varepsilon}-\mathrm{e}^{i t L_{u_0}} \chi_{\varepsilon} \underset{\varepsilon\to 0}{\longrightarrow} 0
\end{align*}
in $L_{+}^2$. Plugging these informations into the formula which gives $f(t,z)$, we have
\begin{align*}
f(t,z) & = \lim _{\varepsilon \rightarrow 0} \frac{1}{2 i \pi}\left\langle\left(U(t)^* G U(t)-z \mathrm{Id}\right)^{-1} \mathrm{e}^{i t L_{u_0}^2} f_0 \mid U(t)^* \chi_{\varepsilon}\right\rangle\\ & = \lim _{\varepsilon \rightarrow 0} \frac{1}{2 i \pi}\left\langle (G-z \mathrm{Id})^{-1} U(t) \mathrm{e}^{i t L_{u_0}^2} f_0 \mid  \chi_{\varepsilon}\right\rangle.
\end{align*}
From (\ref{1.30}), we know that for any $g \in L_{+}^2(\mathbb{R})$, we have
\begin{align*}
g(z) = \frac{1}{2 i \pi} I_{+}\left[(G-z \mathrm{Id})^{-1} g\right] = \lim _{\varepsilon \rightarrow 0} \frac{1}{2 i \pi}\left\langle(G-z \mathrm{Id})^{-1} g \mid \chi_{\varepsilon}\right\rangle.
\end{align*}
So we infer $f(t) = U(t) \mathrm{e}^{i t L_{u_0}^2} f_0 $, thus
\begin{equation}
\label{3.002}
\|f(t)\|_{L^2(\mathbb{R})} = \|f_0\|_{L^2(\mathbb{R})}, \quad \forall t \in \mathbb{R}.
\end{equation}
Now we consider the case of $u_0 \in L_r^2 (\mathbb{R})$. We take $u_0^n \in H_r^2(\mathbb{R}) \underset{n \to \infty}{\longrightarrow} u_0$ in $L^2(\mathbb{R})$ and define
\begin{align*}
f^{n}(t,z) : = \frac{1}{2 i \pi} I_{+}\left(\left(G-2 t L_{u_{0}^{n}}-z\mathrm{Id}\right)^{-1}f_0\right).
\end{align*}
Since we have (\ref{3.002}), there exists a subsequence $f^{n_k}$ such that $f^{n_k}(t) \underset{k \to \infty}{\rightharpoonup} h(t)$ in $L^2(\mathbb{R})$, which implies $f^{n_k}(t,z)$ converges pointwisely to $h(t,z)$ in $\mathbb{C}_{+}$ as $k \to \infty$. Also, from the proof of Theorem \ref{theorem 1.2}, we know that $f^{n_k}(t,z)$
converges pointwisely in $\mathbb{C}_{+}$ to
\begin{align*}
f(t,z) : = \frac{1}{2 i \pi} I_{+}\left(\left(G-2 t L_{u_{0}}-z\mathrm{Id}\right)^{-1}f_0\right)
\end{align*}
as $k \to \infty$. Thus $h$ and $f$ coincide. Since $f$ is the weak limit of $f^{n_k}$ in $L^2(\mathbb{R})$, we know that
\begin{align*}
\|f(t)\|_{L^2(\mathbb{R})} \leq \liminf_{k\to \infty} \|f^{n_k}(t)\|_{L^2(\mathbb{R})} = \|f_0\|_{L^2(\mathbb{R})}, \quad \forall t \in \mathbb{R},
\end{align*}
which implies (\ref{3.001}).
\end{proof}
\begin{remark}
\label{remark 3.02}
From Lemma \ref{lemma 3.001}, we know that
\begin{align*}
|f(t,z)| \leq C(z) \|f_0\|_{L^2(\mathbb{R})}, \quad \forall t \in \mathbb{R}, \, z \in \mathbb{C}_{+},
\end{align*}
where $C(z)$ depends only on $z$.
\end{remark}
Now we consider the equation (\ref{1.02}) with $u_0 \in L_r^2(\mathbb{R})$. By an elementary scaling argument, the solution $u^{\varepsilon}$ of (\ref{1.02}) is given by
\begin{align*}
u^{\varepsilon}(t, x)=\varepsilon v^{\varepsilon}(\varepsilon t, x),
\end{align*}
where $v^{\varepsilon}$ is the solution of the Benjamin–Ono equation (\ref{0.1}) with the initial data 
\begin{align*}
v^{\varepsilon}(0, x)=\frac{1}{\varepsilon} u_{0}(x).
\end{align*}
By applying the explicit formula (\ref{0.2}) to $v^{\varepsilon}$, we infer, for every $z \in \mathbb{C}_{+}$,
\begin{equation}
\label{3.01}
\Pi u^{\varepsilon}(t,z)=\frac{1}{2i\pi}I_{+}\left(\left(G-2 \varepsilon t D + 2t T_{u_{0}}-z \mathrm{Id}\right)^{-1}  \Pi u_{0}\right).
\end{equation}
Formally, we expect the above function converges pointwisely in $\mathbb{C}_{+}$ to
\begin{equation}
\label{3.004}
\frac{1}{2i\pi}I_{+}\left(\left(G+2 t  T_{u_{0}} -z \mathrm{Id}\right)^{-1} \Pi u_{0}\right),
\end{equation}
and such result has been shown in \cite{6} for $u_0 \in L_r^2(\mathbb{R}) \cap L^{\infty}(\mathbb{R})$. In fact, 
we observe that
\begin{align*}
& \left(G+2 t  T_{u_{0}} -z \mathrm{Id}\right) = \left(Id +2 t T_{u_{0}} (G-zId)^{-1}\right)\left(G-zId\right).
\end{align*}
So we may expect that
\begin{equation}
\label{3.02}
\forall f\in L_{+}^2(\mathbb{R}), \quad T_{u_0} (G-zId)^{-1} f \in L_{+}^2(\mathbb{R}).
\end{equation}
We recall the formula (\ref{1.20}),
\begin{align*}
\forall f \in L_{+}^2(\mathbb{R}), \quad \left(G-zId\right)^{-1} f(x) = \frac{f(x)-f(z)}{x-z},
\end{align*}
In fact, for any $z \in \mathbb{C}_{+}$, $\frac{f(z)}{x-z} \in L_x^{\infty}(\mathbb{R})$, so we already have $T_{u_0}\frac{f(z)}{\cdot-z} \in L_{+}^2(\mathbb{R})$. Then we can deduce that (\ref{3.02}) is equivalent to 
\begin{equation}
\label{4.021}
\forall z \in \mathbb{C}_{+}, \quad \forall f \in L_{+}^2(\mathbb{R}), \quad T_{u_0}\frac{f(\cdot)}{\cdot-z} \in L_{+}^{2}(\mathbb{R}).
\end{equation}
Since (\ref{4.021}) holds for all $f \in L_{+}^2(\mathbb{R})$, from (\ref{1.3}), we know that (\ref{4.021}) is equivalent to
\begin{equation}
\label{4.02}
\forall z \in \mathbb{C}_{+}, \quad \frac{u_0(x)}{x-z} \in L_{x}^{\infty}(\mathbb{R}).
\end{equation}
We can also observe that (\ref{4.02}) is equivalent to
\begin{equation}
\label{4.03}
|u_0(x)| \leq C \langle x \rangle \quad \text{ with } \quad\langle x \rangle : = (1+x^2)^{\frac{1}{2}}.
\end{equation}
From the previous arguments, we can deduce that (\ref{4.03}) is a sufficient and necessary condition for \eqref{3.02}. So we may only expect (\ref{1.4}) to hold for initial data in $L_r^2(\mathbb{R})$ which satisfies at least the condition (\ref{4.03}). So far, for $u_0 \in L_r^2(\mathbb{R})$ with $|u_0(x)| \leq C \langle x \rangle$, we cannot show that the zero dispersion limit exists and obtain the formula (\ref{1.4}) for every $t \in \mathbb{R}$, but we can still show that this argument holds for $|t|<\frac{1}{2C}$. Moreover, with $u_0 \in L_r^2(\mathbb{R}) \cap L_{loc}^{\infty} (\mathbb{R})$ satisfying $\lim_{x \to \infty} \frac{|u_0(x)|}{|x|} = 0$, which is a slightly stronger condition than (\ref{4.03}), we can deduce that the zero dispersion limit exists and obtain the formula (\ref{1.4}) for every $t \in \mathbb{R}$.
\begin{remark}
The derivation of \eqref{3.004} is rough. In fact, for $u_0 \in L_r^2(\mathbb{R}) \cap L_{loc}^{\infty} (\mathbb{R})$ satisfying $\lim_{x \to \infty} \frac{|u_0(x)|}{|x|} = 0$, so far we cannot show that the $L^2$ norm of $\left(G-2\varepsilon tD+2 t  T_{u_0}-z \mathrm{Id}\right)^{-1} \Pi u_0$ is uniformly bounded in $\varepsilon$ as $\varepsilon$ tends to 0. Fortunately, Lemma \ref{lemma 3.001} can help us avoid this difficulty, see the proof of Theorem \ref{theorem 1.3} in details.
\end{remark}
Now we deal with the proof of Theorem \ref{theorem 1.3}. To prove Theorem \ref{theorem 1.3}, first we show that $\left(\operatorname{Dom}(G),-iG-2itT_{u_0}\right)$ is maximally dissipative. 
\begin{lemma}
\label{lemma 3.01}
For $u_0 \in L_r^2(\mathbb{R}) \cap L_{loc}^{\infty} (\mathbb{R})$ with $\lim_{x \to \infty} \frac{|u_0(x)|}{|x|} = 0$, $\left(\operatorname{Dom}(G),-iG-2itT_{u_0}\right)$ is maximally dissipative.
\end{lemma}
\begin{proof}
Since $\left(\operatorname{Dom}(G),-iG\right)$ is maximally dissipative, it suffices to prove that, for $0<t<\infty$ fixed, $-2itT_{u_0}$ is infinitesimally small with respect to $G$. It is equivalent to show that, for any $\varepsilon > 0$, we have
\begin{equation}
\label{3.4}
\|T_{u_0} f \|_{L^2} \leq \varepsilon \|Gf\|_{L^2} + C_{\varepsilon}\|f\|_{L^2}, \quad \forall f \in \operatorname{Dom}(G).
\end{equation}
We follow an approach which we used in the proof of Lemma \ref{lemma 3.1}. We recall the definition of $g$,
\begin{align*}
\widehat{g}(\xi) : = \mathbf{1}_{\xi \geq 0} \, \mathrm{e}^{-\xi}.
\end{align*}
Since $u_0$ satisfies $\lim_{x \to \infty} \frac{|u_0(x)|}{|x|} = 0$, then for any $\varepsilon > 0$, there exists $R_{\varepsilon} > 0$ such that
\begin{align*}
\frac{|u_0(x)|}{|x|} < \varepsilon \quad \text{ for all } \quad  |x| \geq R_{\varepsilon}.
\end{align*}
Also, since $u_0 \in L_{loc}^{\infty}(\mathbb{R})$, there exists $M_{\varepsilon} > 0$ such that
\begin{align*}
\|u_0\|_{L^\infty(|x| < R_{\varepsilon})} \leq M_{\varepsilon}.
\end{align*}
Then for $f \in \operatorname{Dom}(G)$, we have
\begin{align*}
& \|T_{u_0} f \|_{L^2(\mathbb{R})}\\ \leq \, & \|u_0 f\|_{L^2(|x|< R_{\varepsilon})}+ \|u_0 \left(f-I_{+}(f) g\right)  \|_{L^2(|x|\geq R_{\varepsilon})}  + \|u_0 \left(I_{+}(f) g \right) \|_{L^2(|x|\geq R_{\varepsilon})}\\ \leq \, & M_{\varepsilon} \left\|f\right\|_{L^2(\mathbb{R})} + \varepsilon \left\|x\left(f-I_{+}(f) g\right)\right\|_{L^2(\mathbb{R})} + \|u_0 \left(I_{+}(f) g \right) \|_{L^2(\mathbb{R})}.
\end{align*}
Since $\mathbf{1}_{\xi\geq 0}\left(\widehat{f}(\xi) - I_{+}(f)\widehat{g}(\xi)\right)$ is continuous at $\xi = 0$, we have
\begin{align*}
\|x (f-I_{+}(f) g)\|_{L^2(\mathbb{R})} = \|G (f-I_{+}(f) g)\|_{L^2(\mathbb{R})}.
\end{align*}
Then  we have
\begin{align*}
\|G (f-I_{+}(f) g)\|_{L^2(\mathbb{R})} & \leq \|Gf\|_{L^2(\mathbb{R})} + |I_{+}(f)| \|Gg\|_{L^2(\mathbb{R})}\\ & \leq \|Gf\|_{L^2(\mathbb{R})} + C\|G f\|_{L^2(\mathbb{R})}^{\frac{1}{2}}\|f\|_{L^{2}(\mathbb{R})}^{\frac{1}{2}}
\end{align*}
and 
\begin{align*}
\|u_0 \left(I_{+}(f) g \right) \|_{L^2(\mathbb{R})} \leq |I_{+}(f)|\|u_0\|_{L^2(\mathbb{R})} \|\hat{g}\|_{L_{\xi}^1(\mathbb{R})} \leq C \|G f\|_{L^2(\mathbb{R})}^{\frac{1}{2}}\|f\|_{L^{2}(\mathbb{R})}^{\frac{1}{2}}.
\end{align*}
Combined with the Young's inequality, we can verify (\ref{3.4}), which implies that $-2itT_{u_0}$ is infinitesimally small with respect to $G$. Then from the Kato-Rellich theorem \ref{theorem 1.5}, we can deduce that $\left(\operatorname{Dom}(G),-iG-2itT_{u_0}\right)$ is maximally dissipative, the proof is complete.
\end{proof}
\begin{remark}
\label{remark 3.2}
Since $\left(\operatorname{Dom}(G),-iG-2itT_{u_0}\right)$ is a maximally dissipative operator, we know that $\left(G+2tT_{u_0}-zId\right)^{-1}$ is well-defined for every $z \in \mathbb{C}_{+}$. By applying (\ref{1.30}), we can deduce that
\begin{align*}
\frac{1}{2 i \pi} I_{+}\left(\left(G+2 t T_{u_{0}}-z \mathrm{Id}\right)^{-1} \Pi u_{0}\right) & = \frac{1}{2 i \pi} I_{+}
\left(\left(G-zId\right)^{-1} \left(Id + 2t T_{u_0} \left(G-zId\right)^{-1}\right)^{-1}\Pi u_0\right) \\ & = \left[\left(Id + 2t T_{u_0} \left(G-zId\right)^{-1}\right)^{-1}\Pi u_0 \right](z)
\end{align*}
is well-defined and holomorphic in $\mathbb{C}_{+}$.
\end{remark}
In Theorem \ref{theorem 1.3}, the point is to prove the existence of the zero dispersion limit and show the formula (\ref{1.4}) of this zero dispersion limit. In the following derivation, we first prove the the second equality of (\ref{1.4}), and then show the existence of the zero dispersion limit.\\\\
To show the second equality of (\ref{1.4}), we need the following integral equality, which has also been introduced in Lemma \ref{lemma 1.7}.
\begin{lemma}
\label{lemma 3.2}
For $f \in L^2(\mathbb{R}) \cap L^{\infty}(\mathbb{R})$ and $n \in \mathbb{N}_{\geq 1}$, we have
\begin{equation}
\label{3.06}
\begin{aligned}
& \int_{\mathbb{R}^n} f(y_1) f(y_2-y_1) ... f(y_n - y_{n-1})f(-y_n) dy_1 dy_2...dy_n \\  = & (n+1) \int_{\{\forall 1\leq j \leq n, y_j > 0\}} f(y_1) f(y_2-y_1) ... f(y_n - y_{n-1}) f(-y_n)dy_1 dy_2...dy_n.
\end{aligned}
\end{equation}
\end{lemma}
\begin{proof}
For $j \in \mathbb{N}_{\geq 0}$ and $0\leq j \leq n$, we define
\begin{align*}
A_{j} : = \{(y_1, y_2, ..., y_n) \in \mathbb{R}^n | \text{ there are $j$ negative elements in }(y_1, y_2, ..., y_n) \}.
\end{align*}
We claim that, for $1\leq j \leq n$, we have
\begin{equation}
\label{3.07}
\begin{aligned}
&\int_{A_0} f(y_1) f(y_2-y_1) ... f(y_n - y_{n-1}) f(-y_n)dy_1 dy_2...dy_n \\ = & \int_{A_j} f(y_1) f(y_2-y_1) ... f(y_n - y_{n-1}) f(-y_n)dy_1 dy_2...dy_n.
\end{aligned}
\end{equation}
We notice that, if we obtain (\ref{3.07}), since the integral on the null set is always equal to 0, we have
\begin{align*}
&\int_{\mathbb{R}^n} f(y_1) f(y_2-y_1) ... f(y_n - y_{n-1})f(-y_n) dy_1 dy_2...dy_n \\ = & \sum_{j= 0}^n \int_{A_j} f(y_1) f(y_2-y_1) ... f(y_n - y_{n-1})f(-y_n) dy_1 dy_2...dy_n \\ = & (n+1)\int_{A_0} f(y_1) f(y_2-y_1) ... f(y_n - y_{n-1})f(-y_n) dy_1 dy_2...dy_n,
\end{align*}
which implies (\ref{3.06}). So the point is to prove (\ref{3.07}).\\\\
Now we prove (\ref{3.07}). For $1\leq i,j \leq n$ and $0\leq k \leq n$, we define
\begin{align*}
B_{k,i,j} : = \{(y_1, y_2, ..., y_n) \in A_k | y_i \text{ is the }j\text{-th} \text{ smallest element} \}.
\end{align*}
For $(y_1, y_2, ..., y_n) \in B_{0,i,j}$, we make the following change of variables
\begin{align*}
\left\{\begin{array}{l} z_{\ell} = y_{\ell+i} - y_i \quad  1 \leq \ell \leq n-i,  \\ z_{n+1-i} = -y_i, \\ z_{\ell} = y_{\ell+i-n-1} - y_i, \quad  n+2-i \leq \ell \leq n. \end{array}\right.
\end{align*}
We notice that $(z_1, z_2,..., z_n) \in B_{j,n+1-i, 1}$, so this linear transformation is from $B_{0,i,j}$ to $B_{j,n+1-i, 1}$, and the absolute value of the determinant of this linear transformation is 1. We also observe that the inverse of this transformation 
\begin{align*}
\left\{\begin{array}{l} y_{k} = z_{k+1+n-i} - z_{n+1-i} \quad  1 \leq k \leq i-1,  \\ y_{i} = -z_{n+1-i}, \\ y_{k} = z_{k-i} - z_{n+1-i}, \quad  i+1 \leq k \leq n. \end{array}\right.
\end{align*}
is from $B_{j,n+1-i, 1}$ to $B_{0,i,j}$, so this transformation is bijective from $B_{0,i,j}$ to $B_{j,n+1-i, 1}$. Then we have
\begin{align*}
&\int_{B_{0,i,j}} f(y_1) f(y_2-y_1) ... f(y_n - y_{n-1}) f(-y_n)dy_1 dy_2...dy_n \\ = & \int_{B_{j,n+1-i, 1}} f(z_1) f(z_2-z_1) ... f(z_n - z_{n-1}) f(-z_n)dz_1 dz_2...dz_n.
\end{align*} 
Combining the above equality, we can deduce that
\begin{align*}
&\int_{A_0} f(y_1) f(y_2-y_1) ... f(y_n - y_{n-1}) f(-y_n)dy_1 dy_2...dy_n\\ = & \sum_{i=1}^n \int_{B_{0,i,j}} f(y_1) f(y_2-y_1) ... f(y_n - y_{n-1}) f(-y_n)dy_1 dy_2...dy_n \\ = & \sum_{i=1}^n \int_{B_{j,n+1-i, 1}} f(z_1) f(z_2-z_1) ... f(z_n - z_{n-1}) f(-z_n)dz_1 dz_2...dz_n \\ = & \int_{A_j} f(z_1) f(z_2-z_1) ... f(z_n - z_{n-1}) f(-z_n)dz_1 dz_2...dz_n,
\end{align*}
which implies (\ref{3.07}). The proof of (\ref{3.06}) is complete.
\end{proof}
\begin{remark}
We notice that the left hand side of (\ref{3.06}) represents the value of the convolution of $(n+1)$-functions $f \in L^2(\mathbb{R}) \cap L^{\infty}(\mathbb{R})$ at the point 0, and the right hand side of (\ref{3.06}) represents the value of the convolution of these $(n+1)$-$f$ restricted in the support of positive half-line at the point 0. As observed in the proof of Lemma \ref{lemma 3.2}, (\ref{3.06}) is derived from (\ref{3.07}), and (\ref{3.07}) is also interesting since it gives the equality between two convolutions at the point 0 with different supports of these $f$.
\end{remark}
Now we are able to prove the second equality of (\ref{1.4}).
\begin{lemma} 
\label{lemma 3.3}
For $u_0 \in L_r^2(\mathbb{R}) \cap L_{loc}^{\infty} (\mathbb{R})$ with $\lim_{x \to \infty} \frac{|u_0(x)|}{|x|} = 0$, we have
\begin{equation}
\label{3.80}
\forall z \in \mathbb{C}_{+}, \quad \frac{1}{2 i \pi} I_{+}\left(\left(G+2 t T_{u_{0}}-z \mathrm{Id}\right)^{-1} \Pi u_{0}\right) = \frac{1}{4i\pi t}\int_{\mathbb{R}} {\rm Log }\left(1 + \frac{2t u_0(y)}{y-z}\right) dy,
\end{equation}
where {\rm Log }denotes the principal value of the logarithm.
\end{lemma}
\begin{proof}
By applying (\ref{1.30}), for any $z \in \mathbb{C}_{+}$, we have
\begin{equation}
\label{3.08}
\begin{aligned}
\frac{1}{2 i \pi} I_{+}\left(\left(G+2 t T_{u_{0}}-z \mathrm{Id}\right)^{-1} \Pi u_{0}\right) & = \frac{1}{2 i \pi} I_{+}
\left(\left(G-zId\right)^{-1} \left(Id + 2t T_{u_0} \left(G-zId\right)^{-1}\right)^{-1}\Pi u_0\right) \\ & = \left[\left(Id + 2t T_{u_0} \left(G-zId\right)^{-1}\right)^{-1}\Pi u_0 \right](z).
\end{aligned}
\end{equation}
Then we only need to show that
\begin{equation}
\forall z \in \mathbb{C}_{+}, \quad  \left[\left(Id + 2t T_{u_0} \left(G-zId\right)^{-1}\right)^{-1}\Pi u_0\right](z) = \frac{1}{4i\pi t}\int_{\mathbb{R}} {\rm Log }\left(1 + \frac{2t u_0(y)}{y-z}\right) dy.
\end{equation}
Since $-2itT_{u_0}$ is infinitesimally small with respect to $G$, we have
\begin{align*}
\left\|2t T_{u_0} (G-zId)^{-1}\right\|_{\mathscr{L}\left(L_{+}^{2}\right)} & \leq \varepsilon \left\|G(G-zId)^{-1}\right\|_{\mathscr{L}\left(L_{+}^{2}\right)} + C_{\varepsilon} \left\|(G-zId)^{-1}\right\|_{\mathscr{L}\left(L_{+}^{2}\right)} \\ & \leq \varepsilon+\frac{C_{\varepsilon}}{\Im(z)}.
\end{align*}
The last inequality above comes from (\ref{1.90}).\\\\
Then we choose $\varepsilon = \frac{1}{4}$ and $z \in i\mathbb{R}_{>0}$ such that $C_{\frac{1}{4}}/\Im(z) < \frac{1}{4}$, and we have
\begin{align*}
\left\|2t T_{u_0} (G-zId)^{-1}\right\|_{\mathscr{L}\left(L_{+}^{2}\right)}< \frac{1}{2}.
\end{align*}
Thus, we can expand $\left(Id + 2t T_{u_0} (G-zId)^{-1}\right)^{-1}$ as a Neumann series for all such $z$. We have
\begin{equation}
\label{3.09}
\left[\left(Id + 2t T_{u_0}(G-zId)^{-1}\right)^{-1} \Pi u_0\right] (z) = \sum_{n=1}^{\infty}(-2t)^{n-1} \left[\left(T_{u_0}\left(G-zId\right)^{-1}\right)^{n-1} \Pi u_0\right](z).
\end{equation}
We recall the formula (\ref{1.10}) for $\Pi u_0(z)$,
\begin{equation}
\label{3.091}
\Pi u_0 (z)=\frac{1}{2i\pi}\int_{\mathbb{R}} \frac{u_0(y)}{y-z} dy,
\end{equation}
which is the case of $n = 1$.\\\\
When $n \geq 2$, we are going to prove
\begin{equation}
\label{1.60}
\left[\left(T_{u_0}\left(G-zId\right)^{-1}\right)^{n-1} \Pi u_0 \right](z) = \frac{1}{2i \pi} \int_{\mathbb{R}} f_z(y) T_{f_z}^{n-2}\Pi f_z(y)dy
\end{equation}
with
\begin{align*}
f_z(y) : = \frac{u_0(y)}{y-z}.
\end{align*}
We now adapt the mathematical induction to deduce (\ref{1.60}). When $n=2$, by applying (\ref{1.10}) and (\ref{1.20}), we have
\begin{align*}
\left[(G-zId)^{-1} \Pi u_0\right](x) & = \frac{\Pi u_0(x)- \Pi u_0 (z)}{x-z} \\ & = \frac{1}{2i \pi(x-z) } \left(\lim_{\delta> 0, \delta \to 0 }\int_{\mathbb{R}} \frac{u_0(y)}{y-x-i\delta} dy - \int_{\mathbb{R}} \frac{u_0(y)}{y-z} dy \right)  \\ & =\frac{1}{2i \pi }\lim_{\delta> 0, \delta \to 0 } \int_{\mathbb{R}} \frac{u_0(y)}{(y-x-i\delta)(y-z)} dy  \\ & = \Pi f_z(x).
\end{align*}
Thus we have
\begin{align*}
\left[T_{u_0}\left(G-zId\right)^{-1} \Pi u_0\right] (z) = \left[T_{u_0} \Pi f_z\right] (z) = \frac{1}{2i \pi}\int_{\mathbb{R}} f_z(y) \Pi f_z(y) dy,
\end{align*}
which yields (\ref{1.60}) with $n = 2$. \\\\
Then we suppose that (\ref{1.60}) holds for $n = k (k \geq 2)$. For $n = k+1$, we have
\begin{align*}
\left[\left(T_{u_0}\left(G-zId\right)^{-1}\right)^{k}\Pi u_0\right](z) = \left[T_{u_0}\left(G-zId\right)^{-1} \left(T_{u_0}\left(G-zId\right)^{-1}\right)^{k-1} \Pi u_0\right](z).
\end{align*}
We note
\begin{align*}
g_{k}(z) := \left[\left(T_{u_0}\left(G-zId\right)^{-1}\right)^{k-1} \Pi u_0\right](z),
\end{align*}
by the assumption, we have
\begin{align*}
g_k (z) =  \frac{1}{2i\pi}\int_{\mathbb{R}} f_z(y) T_{f_z}^{k-2}\Pi f_z(y)dy.
\end{align*}
Then we have
\begin{align*}
&\left(G-zId\right)^{-1} g_k (x)\\ = & \, \frac{g_k(x) -  g_k (z)}{x-z} \\ = & \, \frac{1}{2i\pi (x-z)} \left(\lim_{\delta> 0, \delta \to 0 } \int_{\mathbb{R}} \frac{u_0(y)}{y-x-i\delta} T_{f_z}^{k-2}\Pi f_z(y)dy  - \frac{1}{2i\pi} \int_{\mathbb{R}} \frac{u_0(y)}{y-z} T_{f_z}^{k-2}\Pi f_z(y)dy \right)  \\ = & \, \frac{1}{2i\pi} \lim_{\delta> 0, \delta \to 0 } \int_{\mathbb{R}} \frac{u_0(y)}{(y-x-i\delta)(y-z)} T_{f_z}^{k-2}\Pi f_z(y)dy \\ = & \, T_{f_z}^{k-1}\Pi f_z(x).
\end{align*}
Thus we have
\begin{align*}
\left[\left(T_{u_0}\left(G-zId\right)^{-1}\right)^{k}\Pi u_0\right](z) = \left[T_{u_0}  T_{f_z}^{k-1}\Pi f_z\right](z) = \frac{1}{2i \pi} \int_{\mathbb{R}} f_z(y) T_{f_z}^{k-1}\Pi f_z(y)dy,
\end{align*}
which yields (\ref{1.60}) with $n = k+1$. By the induction, we complete the proof of (\ref{1.60}).\\\\
In fact, we can easily observe that $f_z \in L^1(\mathbb{R}) \cap L^2(\mathbb{R})$, so $\widehat{f}_z \in L^2(\mathbb{R}) \cap L^{\infty}(\mathbb{R})$. Then for $n\geq 2$, by Lemma \ref{lemma 3.2}, we have
\begin{equation}
\label{3.130}
\begin{aligned}
& \int_{\mathbb{R}} f_z(y) T_{f_z}^{n-2}\Pi f_z(y)dy \\= & \, \mathcal{F}_{y \to \eta} \left(f_z T_{f_z}^{n-2}\Pi f_z\right)(0)\\ = & \, \int_{\{\forall 1\leq j \leq n-1, \eta_{j} > 0\}} \widehat{f_z}(\eta_1)\widehat{f_z}(\eta_2-\eta_1)...\widehat{f_z}(\eta_{n-1}-\eta_{n-2}) \widehat{f_z}(-\eta_{n-1})d\eta_1 d\eta_2...d\eta_{n-1}\\ = & \, \frac{1}{n}\int_{\mathbb{R}^{n-1}} \widehat{f_z}(\eta_1)\widehat{f_z}(\eta_2-\eta_1)...\widehat{f_z}(\eta_{n-1}-\eta_{n-2}) \widehat{f_z}(-\eta_{n-1})d\eta_1 d\eta_2...d\eta_{n-1}\\= & \, \frac{1}{n}\mathcal{F}_{y \to \eta} \left(f_z^n\right)(0)  \\ = & \, \frac{1}{n}\int_{\mathbb{R}} f_z^n (y) dy  .
\end{aligned}
\end{equation}
For $t\in \mathbb{R}$ fixed, since $u_0$ satisfies $\lim_{x \to \infty} \frac{|u_0(x)|}{|x|} = 0$, then for any $\varepsilon > 0$, there exists $R_{\varepsilon} > 0$ such that
\begin{align*}
\frac{2|t||u_0(x)|}{|x|} < \varepsilon \quad \text{ for all } \quad  |x| \geq R_{\varepsilon}.
\end{align*}
Also, since $u_0 \in L_{loc}^{\infty}(\mathbb{R})$, there exists $M_{\varepsilon} > 0$ such that
\begin{align*}
2|t|\|u_0\|_{L^\infty(|x| < R_{\varepsilon})} \leq M_{\varepsilon}.
\end{align*}
We fix $\varepsilon = \frac{1}{4}$, and take $z \in i \mathbb{R}_{>0}$ such that $\Im(z) > 4 M_{\frac{1}{4}}$, and then we have
\begin{equation}
\label{3.120}
2|t|\left\|f_z\right\|_{L^{\infty}} < \frac{1}{2}.
\end{equation}
Thus, for $z \in i\mathbb{R}_{>0}$ with $\Im(z)$ large enough, combining (\ref{3.09}), (\ref{3.091}), (\ref{1.60}), (\ref{3.130}) and (\ref{3.120}), we can deduce that
\begin{align*}
& \frac{1}{2 i \pi} I_{+}\left(\left(G+2 t T_{u_{0}}-z \mathrm{Id}\right)^{-1} \Pi u_{0}\right) \\ = & \, \sum_{n=1}^{\infty}(-2t)^{n-1} \left[\left(T_{u_0}\left(G-zId\right)^{-1}\right)^{n-1} \Pi u_0\right](z) \\ = & \, \frac{1}{4i\pi t} \int_{\mathbb{R}} \sum_{n=1}^{\infty} \frac{(-1)^{n-1}}{n} (2tf_z(y))^{n} dy \\ = & \,  \frac{1}{4i\pi t}  \int_{\mathbb{R}}  {\rm Log }(1+2tf_z(y))dy \\ = & \,  \frac{1}{4i\pi t}  \int_{\mathbb{R}}  {\rm Log }(1+\frac{2t u_0}{y-z})dy,
\end{align*}
which implies (\ref{3.80}) for $z \in i\mathbb{R}_{>0}$ with $\Im(z)$ large enough. By Remark \ref{remark 1.70} and Remark \ref{remark 3.2}, we know that the functions (with respect to $z$) on both sides of (\ref{3.80}) are holomorphic in $\mathbb{C}_{+}$, then from the isolated zeros theorem, we can deduce (\ref{3.80}) on the whole upper half-plane $\mathbb{C}_{+}$. The proof is complete.
\end{proof}
\begin{remark}
In fact, (\ref{3.130}) implies that, for every $f \in L^1(\mathbb{R}) \cap L^2(\mathbb{R})$ and for every $n\geq 2$,
\begin{equation}
\label{3.17}
\int_{\mathbb{R}} f(y) T_{f}^{n-2} \Pi f(y) d y = \frac{1}{n} \int_{\mathbb{R}} f^n(y) d y,
\end{equation}
so we have obtained an integral equality (\ref{3.17}) related to the Toeplitz operator $T_{f}$, which is derived from (\ref{3.06}). 
\end{remark}
Combining Lemma \ref{lemma 3.01} and Lemma \ref{lemma 3.3}, we give the following proof of Theorem \ref{theorem 1.3}.\\\\
\textit{Proof of Theorem \ref{theorem 1.3}}. We consider the equation (\ref{1.02}) with $u_0 \in L_r^2(\mathbb{R}) \cap L_{loc}^{\infty} (\mathbb{R})$ satisfying $\lim_{x \to \infty} \frac{|u_0(x)|}{|x|} = 0$. By the $L^2$ conservation law for (\ref{1.02}), we know that
\begin{align*}
\forall t \in \mathbb{R}, \quad \left\|u^{\varepsilon}(t)\right\|_{L^{2}}=\left\|u_{0}\right\|_{L^{2}}.
\end{align*}
Consequently, the family $u^{\varepsilon}(t)$ has weak limits in $L^2(\mathbb{R})$ as $\varepsilon \to 0$. Our task therefore consists in proving that there is only one such weak limit $w_{t}$ . Since $u^{\varepsilon}$ is real valued, so is $w_{t}$, hence $w_{t}=\Pi w_{t}+\overline{\Pi w_{t}}$ on the real line. In Lemma \ref{lemma 3.3}, we have already shown the second equality in (\ref{1.4}), we are therefore reduced to proving the identity
\begin{equation}
\label{3.05}
\forall z \in \mathbb{C}_{+}, \quad \Pi w_{t}(z) =\frac{1}{2 i \pi} I_{+}\left(\left(G+2 t T_{u_{0}}-z \mathrm{Id}\right)^{-1} \Pi u_{0}\right),
\end{equation}
since this identity clearly characterises $w_t$.\\\\
We then recall the formula (\ref{3.01}) for $\Pi u^{\varepsilon}(t,z)$,
\begin{equation}
\Pi u^{\varepsilon}(t,z)=\frac{1}{2i\pi}I_{+}\left(\left(G-2\varepsilon tD+2 t T_{u_{0}} -z \mathrm{Id}\right)^{-1}  \Pi u_{0}\right).
\end{equation}
We define for $\varepsilon \geq 0$,
\begin{align*}
A_{z}^\varepsilon (b): = \left(G-2\varepsilon tD+2 t T_{b} -z \mathrm{Id}\right)^{-1}.
\end{align*}
We have
\begin{align*}
& \Pi u^{\varepsilon}(t,z)- \frac{1}{2i\pi}I_{+}(A_{z}^0 (u_0) \Pi u_0 )\\ = & \, \frac{1}{2i\pi}I_{+}(A_{z}^\varepsilon (u_0) \Pi u_0)-\frac{1}{2i\pi}I_{+}(A_{z}^0 (u_0) \Pi u_0)\\ = & \, \frac{1}{2i\pi} I_{+}\left(A_z^\varepsilon (u_0) (1-\psi(\sqrt{\varepsilon}D)) \Pi u_0 \right) + \frac{1}{2i\pi} I_{+}\left((A_z^\varepsilon (u_0)-A_z^0(u_0))\psi(\sqrt{\varepsilon}D)\Pi u_0\right) \\ - & \, \frac{1}{2i\pi} I_{+}\left(A_z^0(u_0)(1-\psi(\sqrt{\varepsilon}D)) \Pi u_0\right).
\end{align*}
where $\psi \in \mathcal{S}(\mathbb{R})$ with
\begin{align*}
\phi : = \mathcal{F}^{-1}\psi \geq 0, \quad \int_{\mathbb{R}} \phi(x)dx =1 \text{ and } supp\left(\phi\right) \subset \left[-1,1\right].
\end{align*}
In fact, $\phi$ is an approximation identity.\\\\
By Lemma \ref{lemma a.1}, we know that $\|\left(1-\psi(\sqrt{\varepsilon} D)\right) \Pi u_0 \|_{L^2}\underset{\varepsilon\to 0_{+}}{\longrightarrow}$ 0. Then from Lemma \ref{lemma 3.001} and Remark \ref{remark 3.02} (taking $t = \varepsilon t$ and $u_0 = \frac{u_0}{\varepsilon}$ in $(G-2tL_{u_0}-zId)^{-1}$), we know that
\begin{equation}
\label{3.231}
\frac{1}{2i\pi} I_{+}\left(A_z^\varepsilon(u_0) (1-\psi(\sqrt{\varepsilon}D) \Pi u_0  \right) \leq C(z)\|(1-\psi(\sqrt{\varepsilon}D) \Pi u_0\|_{L^2} \underset{\varepsilon \to 0_{+}}{\longrightarrow} 0,
\end{equation}
From Lemma \ref{lemma 3.01}, we can also deduce that
\begin{equation}
\label{3.241}
\frac{1}{2i\pi} I_{+}\left(A_z^0(u_0) (1-\psi(\sqrt{\varepsilon}D) \Pi u_0  \right) \leq C(z)\|(1-\psi(\sqrt{\varepsilon}D) \Pi u_0\|_{L^2} \underset{\varepsilon \to 0_{+}}{\longrightarrow} 0.
\end{equation}
So we only need to analyse the second term above. In fact, from the  second resolvent identity, we have
\begin{align*}
& \frac{1}{2i\pi} I_{+}\left((A_z^\varepsilon (u_0)-A_z^0(u_0))\psi(\sqrt{\varepsilon}D) \Pi u_0\right) \\ = \, &  \frac{1}{2i\pi} I_{+}\left((A_z^\varepsilon (u_0)-A_z^0(\psi(\sqrt{\varepsilon}D)u_0))\psi(\sqrt{\varepsilon}D) \Pi u_0\right) + \frac{1}{2i\pi} I_{+}\left((A_z^0(\psi(\sqrt{\varepsilon}D)u_0)-A_z^0(u_0))\psi(\sqrt{\varepsilon}D)\Pi u_0\right) \\ = \,& \frac{1}{2i\pi} I_{+}\left(A_z^\varepsilon (u_0)\left(2\varepsilon t D + (2tT_{\psi(\sqrt{\varepsilon}D)u_0} - 2tT_{u_0})\right)A_z^0(\psi(\sqrt{\varepsilon}D)u_0)\psi(\sqrt{\varepsilon}D) \Pi u_0\right)\\ + \, & \frac{1}{2i\pi} I_{+}\left(A_z^0 (\psi(\sqrt{\varepsilon}D)u_0)\left(2t T_{u_0}-2tT_{\psi(\sqrt{\varepsilon}D)u_0} \right)A_z^0(u_0)\psi(\sqrt{\varepsilon}D) \Pi u_0\right).
\end{align*}
Now we analyse the function $\psi(\sqrt{\varepsilon} D) u_0$. In fact, we have
\begin{align*}
\psi(\sqrt{\varepsilon} D) u_0(x)=\int_{|y| \leq 1} u_0(x-\sqrt{\varepsilon} y) \phi(y) d y.
\end{align*}
We observe that, for $0<\varepsilon <1$ and $|x| \leq M (M \geq 2)$,
\begin{equation}
\label{3.23}
\|\psi(\sqrt{\varepsilon} D) u_0\|_{L^\infty(|x|\leq M)} \leq \left\|\int_{|y|\leq 1} u_0(x-\sqrt{\varepsilon}y) \phi(y) dy\right\|_{L^\infty(|x|\leq M)}  \leq \|u_0(x)\|_{L^\infty(|x|\leq M+1)}.
\end{equation}
Also, for $0<\varepsilon <1$ and $|x| > M (M \geq 2)$, we have
\begin{equation}
\label{3.24}
\begin{aligned}
\sup_{|x|\geq M}\frac{|\psi(\sqrt{\varepsilon} D) u_0(x)|}{|x|} & \leq \sup_{|x|\geq M}\frac{\int_{|y|\leq 1} |u_0(x-\sqrt{\varepsilon}y)| \phi(y) dy}{|x|}  \\ & = \sup_{|x|\geq M}\frac{\int_{|y|\leq 1} |u_0(x-\sqrt{\varepsilon}y)| \phi(y) dy}{|x-\sqrt{\varepsilon}y|} \frac{|x-\sqrt{\varepsilon}y|}{|x|} \\ & \leq 2\sup_{|x|\geq M-1}\frac{|u_0(x)|}{|x|} \underset{M\to +\infty}{\longrightarrow} 0.
\end{aligned}
\end{equation}
Based on the above facts, we observe that the bound of $\psi(\sqrt{\varepsilon} D) u_0 \in L_{loc}^\infty(\mathbb{R})$ and the convergence of $\frac{|\psi(\sqrt{\varepsilon} D) u_0|}{|x|} \underset{x \to \infty}{\longrightarrow} 0$ hold uniformly in $\varepsilon$. Thus, from the proof of Lemma \ref{lemma 3.01}, we can deduce that
\begin{equation}
\label{3.25}
\|A_z^0\left(\psi(\sqrt{\varepsilon} D) u_0\right) f\|_{L^2(\mathbb{R})} \leq C(z) \|f\|_{L^2(\mathbb{R})}, \quad \forall f \in L_{+}^2(\mathbb{R})
\end{equation}
and
\begin{equation}
\label{3.26}
\|G A_z^0\left(\psi(\sqrt{\varepsilon} D) u_0\right) f\|_{L^2(\mathbb{R})} \leq C(z) \|f\|_{L^2(\mathbb{R})}, \quad \forall f \in L_{+}^2(\mathbb{R}),
\end{equation}
where $C(z)$ is a constant which depends on $z$ (not on $\varepsilon$).\\\\
Now we want to show that 
\begin{equation}
\label{3.2701}
 h_{\varepsilon}: = A_z^0\left(\psi(\sqrt{\varepsilon} D) u_0\right) \psi(\sqrt{\varepsilon} D) \Pi u_0 \underset{\varepsilon \to 0_{+}}{\longrightarrow} h_0: = A_z^0\left(u_0\right)  \Pi u_0  \text{ in } L_{+}^2(\mathbb{R}).
\end{equation}
In fact, we have
\begin{align*}
& \quad A_z^0\left(\psi(\sqrt{\varepsilon} D) u_0\right) \psi(\sqrt{\varepsilon} D) \Pi u_0  -A_z^0\left(u_0\right)  \Pi u_0 \\ & =A_z^0\left(\psi(\sqrt{\varepsilon} D) u_0\right)\left(\psi(\sqrt{\varepsilon} D) -1\right)\Pi u_0 + \left(A_z^0\left(\psi(\sqrt{\varepsilon} D) u_0\right)-A_z^0\left(u_0\right)\right)\Pi u_0.
\end{align*}
Combine (\ref{3.25}) and Lemma \ref{lemma a.1}, we can deduce that
\begin{align*}
\|A_z^0\left(\psi(\sqrt{\varepsilon} D) u_0\right)\left(\psi(\sqrt{\varepsilon} D) -1\right)\Pi u_0\|_{L^2(\mathbb{R})} \underset{\varepsilon \to 0_{+}}{\longrightarrow} 0.
\end{align*}
From the second resolvent identity and \eqref{3.25}, we have
\begin{align*}
& \quad \left\| \left(A_z^0\left(\psi(\sqrt{\varepsilon} D) u_0\right)-A_z^0\left(u_0\right)\right)\Pi u_0 \right\|_{L^2(\mathbb{R})} \\ & = 2|t|\left\|A_z^0\left(\psi(\sqrt{\varepsilon} D) u_0\right)(T_{u_0} - T_{\psi(\sqrt{\varepsilon} D) u_0})A_z^0\left(u_0\right)\Pi u_0 \right\|_{L^2(\mathbb{R})} \\ & \leq C(t,z) \left\|(T_{u_0} - T_{\psi(\sqrt{\varepsilon} D) u_0})h_0\right\|_{L^2(\mathbb{R})}.
\end{align*}
To show \eqref{3.2701}, we only need to prove $\left\|(T_{u_0} - T_{\psi(\sqrt{\varepsilon} D) u_0})h_0\right\|_{L^2(\mathbb{R})} \underset{\varepsilon \to 0_{+}}{\longrightarrow} 0$.\\\\
From the assumptions of $\frac{|u_0(x)|}{|x|} \underset{x \to \infty}{\longrightarrow}0$ and \eqref{3.24}, we know that, for any $\delta >0$, there exists $M_{\delta} > 0$ such that for every $\varepsilon \geq 0$,
\begin{align*}
\frac{|\psi(\sqrt{\varepsilon} D) u_0(x)|}{|x|} \leq \delta, \quad \text{ for } |x| \geq M_{\delta}.
\end{align*}
From the proof of Lemma \ref{lemma 3.01}, we infer
\begin{align*}
& \quad \|(T_{u_0}-T_{\psi(\sqrt{\varepsilon} D) u_0}) h_0\|_{L^2(\mathbb{R})} \\ & \leq \left\|(1-\psi(\sqrt{\varepsilon} D))u_0 h_0\right\|_{L^2\left(|x|<M_{\delta}\right)} \\ &+ \left\|\frac{(1-\psi(\sqrt{\varepsilon} D))u_0(x)}{x}\right\|_{L_x^\infty(|x|\geq M_{\delta})}\left\|x\left(h_0-I_{+}(h_0) g\right)\right\|_{L^2\left(|x| \geq M_{\delta}\right)}\\ &+\left\|(1-\psi(\sqrt{\varepsilon} D))u_0\left(I_{+}(h_0) g\right)\right\|_{L^2(\mathbb{R})} \\ &\leq \left\|(1-\psi(\sqrt{\varepsilon} D))u_0 h_0\right\|_{L^2\left(|x|<M_{\delta}\right)} \\ & + (C\delta+ C \|1-\psi(\sqrt{\varepsilon} D))u_0\|_{L^2(\mathbb{R})})(\|h_0\|_{L^2(\mathbb{R})}  +\|Gh_0\|_{L^2(\mathbb{R})}),
\end{align*}
where
\begin{align*}
\widehat{g}(\xi) : = \mathbf{1}_{\xi \geq 0} \, \mathrm{e}^{-\xi}.
\end{align*}
From Lemma \ref{lemma a.2}, we know that
\begin{align*}
\left\|(1-\psi(\sqrt{\varepsilon} D))u_0 h_0\right\|_{L^2\left(|x|<M_{\delta}\right)} \underset{\varepsilon \to 0_{+}}{\longrightarrow} 0.
\end{align*}
So we can deduce that $\|(T_{u_0}-T_{\psi(\sqrt{\varepsilon} D) u_0}) h_0\|_{L^2(\mathbb{R})} \underset{\varepsilon \to 0_{+}}{\longrightarrow} 0$, and thus \eqref{3.2701} follows.\\
Now we have
\begin{align*}
& \quad \|(T_{u_0}-T_{\psi(\sqrt{\varepsilon} D) u_0}) A_z^0\left(\psi(\sqrt{\varepsilon} D) u_0\right) \psi(\sqrt{\varepsilon} D) \Pi u_0\|_{L^2(\mathbb{R})}\\ & = \|(T_{u_0}-T_{\psi(\sqrt{\varepsilon} D) u_0}) h_\varepsilon\|_{L^2(\mathbb{R})} \\ & \leq \left\|(1-\psi(\sqrt{\varepsilon} D))u_0 h_\varepsilon\right\|_{L^2\left(|x|<M_{\delta}\right)} \\ &+ \left\|\frac{(1-\psi(\sqrt{\varepsilon} D))u_0(x)}{x}\right\|_{L_x^\infty(|x|\geq M_{\delta})}\left\|x\left(h_0-I_{+}(h_\varepsilon) g\right)\right\|_{L^2\left(|x| \geq M_{\delta}\right)}\\ &+\left\|(1-\psi(\sqrt{\varepsilon} D))u_0\left(I_{+}(h_\varepsilon) g\right)\right\|_{L^2(\mathbb{R})} \\ &\leq \left\|(1-\psi(\sqrt{\varepsilon} D))u_0 h_\varepsilon\right\|_{L^2\left(|x|<M_{\delta}\right)} \\ & + (C\delta+ C \|1-\psi(\sqrt{\varepsilon} D))u_0\|_{L^2(\mathbb{R})})(\|h_\varepsilon\|_{L^2(\mathbb{R})}  +\|Gh_\varepsilon\|_{L^2(\mathbb{R})}).
\end{align*}
Since we have \eqref{3.2701}, from Corollary \ref{corollary a.3}, we infer
\begin{align*}
\left\|(1-\psi(\sqrt{\varepsilon} D))u_0 h_\varepsilon\right\|_{L^2\left(|x|<M_{\delta}\right)} \underset{\varepsilon \to 0_{+}}{\longrightarrow} 0.
\end{align*}
Then combine Lemma \ref{lemma 3.001}, Lemma \ref{lemma a.1}, (\ref{3.25}) and (\ref{3.26}), we can deduce that
\begin{equation}
\label{3.271}
 \frac{1}{2i\pi} I_{+}\left(A_z^\varepsilon (u_0)\left( 2tT_{\psi(\sqrt{\varepsilon}D)u_0} - 2tT_{u_0}\right)A_z^0(\psi(\sqrt{\varepsilon}D)u_0)\psi(\sqrt{\varepsilon}D) \Pi u_0\right) \underset{\varepsilon\to 0_{+}}{\longrightarrow} 0.
\end{equation}
Similarly, we can also infer 
\begin{equation}
\label{3.281}
\frac{1}{2i\pi} I_{+}\left(A_z^0 (\psi(\sqrt{\varepsilon}D)u_0)\left(2t T_{u_0}-2tT_{\psi(\sqrt{\varepsilon}D)u_0} \right)A_z^0(u_0)\psi(\sqrt{\varepsilon}D) \Pi u_0\right) \underset{\varepsilon\to 0_{+}}{\longrightarrow} 0.
\end{equation}
Now we only need to deal with
\begin{align*}
\frac{1}{2i\pi} I_{+}\left(A_z^\varepsilon (u_0)\left(2\varepsilon t D \right)A_z^0(\psi(\sqrt{\varepsilon}D)u_0)\psi(\sqrt{\varepsilon}D) \Pi u_0\right).
\end{align*}
We recall the formula
\begin{align*}
[B,A^{-1}] = -A^{-1} [B,A] A^{-1},
\end{align*}
so we have
\begin{align*}
& [2\varepsilon t D,A_z^0 (\psi(\sqrt{\varepsilon}D)u_0)] \psi(\sqrt{\varepsilon}D) \Pi u_0 \\  = &\, -A_z^0 (\psi(\sqrt{\varepsilon}D)u_0) [2\varepsilon t D, A_z^0 (\psi(\sqrt{\varepsilon}D)u_0)^{-1}]A_z^0 (\psi(\sqrt{\varepsilon}D)u_0) \psi(\sqrt{\varepsilon}D) \Pi u_0 \\  = &\,A_z^0 (\psi(\sqrt{\varepsilon}D)u_0) \left(2i\varepsilon t - 4\sqrt{\varepsilon} t^2 T_{\sqrt{\varepsilon}D\psi(\sqrt{\varepsilon}D)u_0 }\right)A_z^0 (\psi(\sqrt{\varepsilon}D)u_0) \psi(\sqrt{\varepsilon}D)\Pi u_0.
\end{align*}
In fact, we can show that the bound of $\sqrt{\varepsilon}D\psi(\sqrt{\varepsilon}D)u_0 \in L_{loc}^{\infty}(\mathbb{R})$ and the convergence of $\frac{|\sqrt{\varepsilon}D\psi(\sqrt{\varepsilon}D)u_0(x)|}{|x|} \underset{x\to \infty}{\longrightarrow} 0$ hold uniformly in $\varepsilon$. Similarly as before, we can deduce that 
\begin{equation}
\label{3.27}
\|[2\varepsilon t D, A_z^0 (\psi(\sqrt{\varepsilon}D)u_0)] \psi(\sqrt{\varepsilon}D) \Pi u_0\|_{L^2(\mathbb{R})} \underset{\varepsilon \to 0_{+}}{\longrightarrow} 0.
\end{equation}
Since 
\begin{align*}
\|\sqrt{\varepsilon}D \psi(\sqrt{\varepsilon}D) \Pi u_0)\|_{L^2(\mathbb{R})} \leq C \|\Pi u_0\|_{L^2(\mathbb{R})}
\end{align*}
with $C$ not depending on $\varepsilon$, we can also deduce that 
\begin{equation}
\label{3.28}
\left\|A_z^0 (\psi(\sqrt{\varepsilon}D)u_0) (2\varepsilon t D)\psi(\sqrt{\varepsilon}D) \Pi u_0\right\|_{L^2(\mathbb{R})} \underset{\varepsilon \to 0_{+}}{\longrightarrow} 0.
\end{equation}
Combine Lemma \ref{lemma 3.001}, (\ref{3.27}) and (\ref{3.28}), we infer
\begin{equation}
\label{3.31}
\frac{1}{2i\pi} I_{+}\left(A_z^\varepsilon (u_0)\left(2\varepsilon t D \right)A_z^0(\psi(\sqrt{\varepsilon}D)u_0)\psi(\sqrt{\varepsilon}D) \Pi u_0\right) \underset{\varepsilon \to 0_{+}}{\longrightarrow} 0.
\end{equation}
Thus from \eqref{3.231}, \eqref{3.241}, \eqref{3.271}, \eqref{3.281} and \eqref{3.31}, we conclude that, for every $z \in \mathbb{C}_{+}$,
\begin{align*}
\Pi u^{\varepsilon}(t,z)- \frac{1}{2i\pi}I_{+}(A_{z}^0 (u_0) \Pi u_0) \underset{\varepsilon \to 0_{+}}{\longrightarrow} 0.
\end{align*}
Also, from the weak convergence of $u^{\varepsilon}(t)$ to $w_{t}$ in $L^2(\mathbb{R})$, we have
\begin{align*}
\forall z \in \mathbb{C}_{+}, \quad \Pi u^\varepsilon(t,z)-\Pi w_{t}(z) = \int_{0}^{\infty} \mathrm{e}^{iz\xi} \left(\widehat{u^{\varepsilon}}(t,\xi)-\widehat{w}_t(\xi)\right) d\xi \underset{\varepsilon \to 0_{+}}{\longrightarrow} 0, 
\end{align*}
and thus (\ref{3.05}) follows. The proof is complete.
\begin{remark} In fact, for $u_0 \in L_r^2(\mathbb{R})$ with $|u_0(x)| \leq C \langle x\rangle$ and for $|t| < \frac{1}{2C}$, by applying the method in the proof of Lemma \ref{lemma 3.01}, we can deduce that $-2itT_{u_0}$ is $G$-bounded with the relative bound smaller than 1. Then by the Kato-Rellich theorem \ref{theorem 1.5}, we can conclude that $\left(\operatorname{Dom}(G),-i G-2 i t T_{u_0}\right)$ is maximally dissipative. With a slight modification of the proof of Lemma \ref{lemma 3.3}, we can also show (\ref{3.80}) for $u_0 \in L_r^2(\mathbb{R})$ with $|u_0(x)| \leq C \langle x\rangle$ in a short time range $|t| < \frac{1}{2C}$. Finally, by following the same approach used in the proof of Theorem \ref{theorem 1.3}, we can deduce Corollary \ref{corollary 1.4}.
\end{remark}
\section{Final comments and open problems}
\label{section 4}
Let us briefly give some comments related to the previous sections.\\\\
1. Recently, R. Killip, T. Laurens and M. Vişan have extended continuously the flow map of (\ref{0.1}) to $u_0 \in H_r^{s}(\mathbb{R})$ with $-\frac{1}{2} < s < 0$ \cite{7}. However, so far we have not been able to extend the explicit formula (\ref{0.2}) to $u_0 \in H_r^{s}(\mathbb{R})$ with $-\frac{1}{2} < s < 0$. In fact, we cannot apply directly the perturbation argument used in Section \ref{section 2} to this case, and we do not know if $\left(G-2 t L_{u_0}-z \mathrm{Id}\right)^{-1}$ exists on $H_{+}^s(\mathbb{R})$ with $-\frac{1}{2} < s < 0$. So far no other suitable approach has been found to give an explicit formula for the solution of (\ref{0.1}) in this case. We remark that we have also the global well-posedness of the Benjamin–Ono equation on the torus in $H_r^{s}(\mathbb{T})$ with $-\frac{1}{2} < s < 0$ \cite{8}\cite{7}, and the explicit formula for the Benjamin–Ono equation on the torus has been successfully extended to $u_0 \in H_r^{s}(\mathbb{T})$ with $-\frac{1}{2} < s < 0$ \cite{1}. \\\\
2. As explained in Remark \ref{remark 1.70}, we know that the expression
\begin{align*}
\frac{1}{4i\pi t}\int_{\mathbb{R}} {\rm Log }\left(1 + \frac{2t u_0(y)}{y-z}\right) dy
\end{align*}
makes sense if $u_0 \in L_r^2(\mathbb{R})$. Nevertheless, this does not imply that the zero dispersion limit exists in this case. For $u_0 \in L_{r}^2(\mathbb{R})$, let $u^{\varepsilon}$ be the corresponding solution to (\ref{1.02}) with the initial data $u_0$, and we take a sequence $u_0^n \in L_{r}^2(\mathbb{R})\cap L^{\infty}(\mathbb{R})$ which converges to $u_0$ in $L^2(\mathbb{R})$. In fact, from (\ref{1.4}) and Remark \ref{remark 1.70}, we know that
\begin{align*}
\lim_{n\to \infty} \lim_{\varepsilon \to 0}\Pi u_n^{\varepsilon} (t, z) = \lim_{n\to \infty} \frac{1}{4i\pi t}\int_{\mathbb{R}} {\rm Log }\left(1 + \frac{2t u_0^n(y)}{y-z}\right) dy = \frac{1}{4i\pi t}\int_{\mathbb{R}} {\rm Log }\left(1 + \frac{2t u_0(y)}{y-z}\right) dy,
\end{align*}
where $u_n^\varepsilon$ denotes the corresponding solution to (\ref{1.02}) with the initial data $u_0^n$.\\\\
To show the existence of the zero dispersion limit with the initial data $u_0 \in L_r^2(\mathbb{R})$, we only need to show that $\lim_{\varepsilon \to 0} \lim_{n \to \infty}\Pi u_n^{\varepsilon} (t, z)$ exists. A natural idea is to show that these two limits can be exchanged in order, which would then imply that 
\begin{align*}
\lim_{\varepsilon \to 0}\Pi u^{\varepsilon} (t, z)  =\lim_{\varepsilon \to 0} \lim_{n \to \infty}\Pi u_n^{\varepsilon} (t, z) = \lim_{n\to \infty} \lim_{\varepsilon \to 0}\Pi u_n^{\varepsilon} (t, z) = \frac{1}{4i\pi t}\int_{\mathbb{R}} {\rm Log }\left(1 + \frac{2t u_0(y)}{y-z}\right) dy.
\end{align*}
However, we lack certain uniform conditions for this double limit to prove the order exchangeability, so the existence for the zero dispersion limit with the initial data $u_0 \in L_r^2(\mathbb{R})$ is still unknown even in a short time.\\\\
Also, as observed in (\ref{4.03}), the condition 
\begin{align*}
u_0 \in L_r^2(\mathbb{R}) \text{ with } \left|u_{0}(x)\right| \leq C\langle x\rangle
\end{align*}
is a sufficient and necessary condition for \eqref{3.02}. With this condition we can only deduce the existence of the zero dispersion limit in a short time, and the existence of the zero dispersion limit in a long time is still unknown for the same reason explained in the previous paragraph. A natural idea to solve this difficulty is to apply the Kato-Rellich theorem to show that 
\begin{align*}
\left(G+2 t T_{u_{0}}-z \mathrm{Id}\right)^{-1}
\end{align*}
exists on $L_{+}^2(\mathbb{R})$ for every $z \in \mathbb{C}_{+}$, but the perturbation argument fails in a long time range since we cannot deduce that the relative bound of $-2itT_{u_0}$ with respect to $G$ is smaller than 1 for every $t \in \mathbb{R}$.\\\\
3. The zero dispersion limit for the Benjamin–Ono equation on the torus was studied by L. Gassot in \cite{14}\cite{15}. In \cite{15}, the explicit formula for the Benjamin–Ono equation on the torus established in \cite{1} was used to prove the existence of the zero dispersion limit for every initial datum in $L^{\infty}(\mathbb{T})$. The existence of the zero–dispersion limit for more singular initial data is still an open problem. As introduced in Remark \ref{remark 1.6}, L. Gassot has also obtained the formula (\ref{1.9}) in the special case of a general bell shaped initial datum in \cite{14}\cite{15}.
\begin{appendix}
\section{Appendix}
In the appendix, we recall the definition of $\psi \in \mathcal{S}(\mathbb{R})$ with
\begin{align*}
\phi : = \mathcal{F}^{-1}\psi \geq 0, \quad \int_{\mathbb{R}} \phi(x)dx =1 \text{ and } supp\left(\phi\right) \subset \left[-1,1\right].
\end{align*}
Here $\phi$ is an approximation identity. Then we introduce the following approximation result in $L^2(\mathbb{R})$.
\begin{lemma}
\label{lemma a.1}
Given $f \in L^2(\mathbb{R})$, we have
\begin{align*}
\left\|f-\psi(\varepsilon D) f\right\|_{L^2(\mathbb{R})} \underset{\varepsilon \to 0_{+}}{\longrightarrow} 0.
\end{align*}
\end{lemma}
\begin{proof}
we know that
\begin{align*}
\psi(\varepsilon D) f (x)= \int_{|y|\leq 1} f(x-\varepsilon y) \phi(y) dy,
\end{align*}
so
\begin{align*}
& \quad \left\|f-\psi(\varepsilon D) f\right\|_{L^2(\mathbb{R})}^2 \\ &  = \int_{\mathbb{R}} \left|f(x)-\psi(\varepsilon D) f (x)\right|^2 dx \\ & \leq \int_{\mathbb{R}} \left(\int_{|y|\leq 1} \left|f(x)-f(x-\varepsilon y)\right|\phi(y) dy\right)^2 dx \\ & \leq \int_{|y|\leq 1} \phi(y)^2 dy \int_{\mathbb{R}} \int_{|y|\leq 1} \left|f(x)-f(x-\varepsilon y)\right|^2 dy dx \\ & \leq C  \int_{|y|\leq 1}\int_{\mathbb{R}} \left|f(x)-f(x-\varepsilon y)\right|^2 dx dy \\ & \leq C  \sup_{|y|\leq 1}\int_{\mathbb{R}} \left|f(x)-f(x-\varepsilon y)\right|^2 dx \underset{\varepsilon \to 0_{+}}{\longrightarrow} 0.
\end{align*}
The second inequality above comes from the Cauchy-Schwarz inequality, and the last inequality above comes from the continuity of translations in $L^2(\mathbb{R})$.
\end{proof}
Also, we have the following approximation result. 
\begin{lemma}
\label{lemma a.2}
Given $g \in L_{loc}^{\infty}(\mathbb{R})\cap L^2(\mathbb{R})$ and $f \in L^2(\mathbb{R})$, then for every $M >0$, we have 
\begin{equation}
\label{A.1}
\left\|(1-\psi(\varepsilon D))g f\right\|_{L^2(|x|<M)} \underset{\varepsilon \to 0_{+}}{\longrightarrow} 0.
\end{equation}
\end{lemma}
\begin{proof}
In fact, we have
\begin{align*}
& \quad \left\|(1-\psi(\varepsilon D))g f\right\|_{L^2(|x|<M)}^2 \\ & = \int_{|x|<M} |f(x)|^2 \left|\int_{|y|\leq 1}  \left(g(x)-g(x-\varepsilon y)\right)\phi(y) dy\right|^2dx \\ & \leq \int_{|y|\leq 1} \phi(y)^2 dy \int_{|y|\leq 1} \int_{|x|<M} |f(x)|^2 |g(x)-g(x-\varepsilon y)|^2 dx dy \\ & \leq C \int_{|y|\leq 1} \int_{|x|<M} |f(x)|^2 |g(x)-g(x-\varepsilon y)|^2 dx dy.
\end{align*}
The first inequality comes from the Cauchy-Schwarz inequality.\\\\
Since $g \in L_{loc}^\infty(\mathbb{R})$, we can deduce that for every $0<\varepsilon<1$ and for every $|x|< M, |y|\leq1$, we have $|g(x)-g(x-\varepsilon y)|^2 \leq C_M$.\\\\
Also, as $f \in L^2(\mathbb{R})$, for every $\delta>0$, there exists $h_{\delta} \in L^\infty(\mathbb{R})$ with $\|h_\delta\|_{L^{\infty}(\mathbb{R})}^2 \leq C_{\delta}$ such that $\|f-h_{\delta}\|_{L^2(\mathbb{R})}^2 \leq \delta$. Then we have
\begin{align*}
& \quad\int_{|y|\leq 1} \int_{|x|<M} |f(x)|^2 |g(x)-g(x-\varepsilon y)|^2 dx dy \\ & \leq 2 \quad\int_{|y|\leq 1} \int_{|x|<M} |h_\delta(x)|^2 |g(x)-g(x-\varepsilon y)|^2 dx dy \\ & + 2\quad\int_{|y|\leq 1} \int_{|x|<M} |f(x)-h_\delta(x)|^2 |g(x)-g(x-\varepsilon y)|^2 dx dy \\ & \leq C_{\delta}\sup_{|y|\leq 1}\int_{|x|<M}|g(x)-g(x-\varepsilon y)|^2 dx +  C_M  \delta.
\end{align*}
By the continuity of translations in $L^2(\mathbb{R})$, we know that
\begin{align*}
\sup_{|y|\leq 1}\int_{|x|<M}|g(x)-g(x-\varepsilon y)|^2 dx \underset{\varepsilon \to 0_{+}}{\longrightarrow} 0.
\end{align*}
Thus we infer \eqref{A.1}. 
\end{proof}
Lemma \ref{lemma a.2} allows us to deduce the following corollary.
\begin{corollary}
\label{corollary a.3}
Given $g \in L_{loc}^{\infty}(\mathbb{R})\cap L^2(\mathbb{R})$. Assume that  $f_{\varepsilon} \in L^2(\mathbb{R})$ satisfying $\|f_{\varepsilon}-f\|_{L^2(\mathbb{R})} \underset{\varepsilon \to 0_{+}}{\longrightarrow} 0$. Then for every $M >0$, we have 
\begin{equation}
\label{A.2}
\left\|(1-\psi(\varepsilon D))g f_{\varepsilon}\right\|_{L^2(|x|<M)} \underset{\varepsilon \to 0_{+}}{\longrightarrow} 0.
\end{equation}
\end{corollary} 
\begin{proof}
By Lemma \ref{lemma a.2} we know that
\begin{align*}
\left\|(1-\psi(\varepsilon D))g f\right\|_{L^2(|x|<M)} \underset{\varepsilon \to 0_{+}}{\longrightarrow} 0.
\end{align*}
From the proof of Lemma \ref{lemma a.2}, we have
\begin{align*}
& \quad \|(1-\psi(\varepsilon D)) g (f_\varepsilon - f)\|_{L^2(|x|<M)}^2 \\ & \leq C \int_{|y|\leq 1} \int_{|x|<M} |f_\varepsilon(x) - f(x)|^2 |g(x)-g(x-\varepsilon y)|^2 dx dy \\ & \leq  C_M \int_{|x|<M} |f_\varepsilon(x) - f(x)|^2 dx \underset{\varepsilon \to 0_{+}}{\longrightarrow} 0.
\end{align*}
Thus we conclude \eqref{A.2}.
\end{proof}
\end{appendix}

\end{document}